\documentclass[aap,MSNbibl,dvips]{arximspdf}

\doi{10.1214/12-AAP905} %kopijuoti is PTS
\volume{23}
\issue{6}
\pubyear{2013}
\firstpage{2382}
\lastpage{2419}

\makeatletter

\newcommand{\cal}{\mathcal}

\newcommand{\pr}{\mathbb{P}}
\newcommand{\E}{\mathbb{E}}

\newcommand{\reals}{{\mathbb R}}

\newtheorem{theorem}{Theorem}
\newtheorem{lemma}{Lemma}
\newtheorem{proposition}{Proposition}
\newtheorem{corollary}{Corollary}

\makeatother

\begin{document}
\begin{frontmatter}

\title{Steady-state $\mathit{GI}/\mathit{GI}/\mathbf{\mathit{\lowercase{n}}}$ queue in the Halfin--Whitt regime}
\runtitle{Steady-state $\mathit{GI}/\mathit{GI}/\lowercase{n}$ queue in the Halfin--Whitt regime}

\begin{aug}
\author[A]{\fnms{David} \snm{Gamarnik}\thanksref{t1}\ead[label=e1]{gamarnik@mit.edu}}
\and
\author[B]{\fnms{David A.} \snm{Goldberg}\corref{}\ead[label=e2]{dgoldberg9@isye.gatech.edu}}
\runauthor{D. Gamarnik and D. A. Goldberg}
\affiliation{MIT and Georgia Institute of Technology}
\address[A]{Operations Research Center\\
\quad and Sloan School of Management\\
MIT\\
Cambridge, Massachusetts 02139\\
USA\\
\printead{e1}} %adresu isvedimo komanda gale!
\address[B]{Georgia Institute of Technology\\
Atlanta, Georgia 30332\\
USA\\
\printead{e2}}
\end{aug}

\thankstext{t1}{Supported by NSF Grant CMMI-0726733.}

% HISTORY:
\received{\smonth{3} \syear{2011}}
\revised{\smonth{9} \syear{2012}}

% ABSTRACT
%
\begin{abstract}
We consider the FCFS $\mathit{GI}/\mathit{GI}/n$ queue in the so-called Halfin--Whitt
heavy traffic regime. We prove that under minor technical conditions
the associated sequence of steady-state queue length distributions,
normalized by $n^{{1/2}}$, is tight. We derive an upper bound on
the large deviation exponent of the limiting steady-state queue length
matching that conjectured by Gamarnik and Momcilovic
[\textit{Adv. in Appl. Probab.} \textbf{40} (2008) 548--577].
We also prove a matching lower bound when the arrival process is
Poisson.

Our main proof technique is the derivation of new and simple bounds for
the FCFS $\mathit{GI}/\mathit{GI}/n$ queue. Our bounds are of a structural nature, hold
for all $n$ and all times $t \geq0$, and have intuitive closed-form
representations as the suprema of certain natural processes which
converge weakly to Gaussian processes. We further illustrate the
utility of this methodology by deriving the first nontrivial bounds
for the weak limit process studied in [\textit{Ann. Appl. Probab.}
\textbf{19} (2009) 2211--2269].
\end{abstract}

% KEYWORDS

\end{frontmatter}

%s1 #&#
\section{Introduction}\label{SectionIntroduction}
Parallel server queueing systems can operate in a variety of regimes
that balance between efficiency
and quality of offered service. This is captured by the so-called
Halfin--Whitt (H--W) heavy traffic regime, which can be described as
critical with respect to
the probability that an arriving customer has to wait for service.
Namely, in this regime the stationary probability of wait is bounded
away from both 0 and 1, as the number of servers grows.
Although studied originally by Erlang \cite{E48} and Jagerman \cite
{J74}, the regime was formally introduced by Halfin and Whitt \cite
{HW81}, who
studied the $\mathit{GI}/M/n$ system (for large $n$) when the traffic intensity
scales like $1 - Bn^{-{1/2}}$ for some strictly positive $B$.
Namely, the parameter $B$ controls how close to overloaded the system
is in heavy traffic.
They proved that under minor technical assumptions on the inter-arrival
distribution, this sequence of $\mathit{GI}/M/n$ queueing models has the
following properties:
\begin{longlist}
\item the steady-state probability that an arriving job finds all
servers busy (i.e., the probability of wait) has a nontrivial limit;
\item the sequence of queueing processes, normalized by $n^{{1/2}}$,
converges weakly to a nontrivial positive recurrent diffusion;
\item the sequence of steady-state queue length distributions,
normalized by $n^{{1/2}}$, is tight and converges
distributionally to the mixture of a point mass at $0$ and an
exponential distribution.
\end{longlist}
Furthermore, this steady-state probability of wait can be parametrized
as a function of $B$, with larger values of $B$ corresponding to
smaller probabilities of wait.
Similar weak convergence results under the H--W scaling were
subsequently obtained for more general multi-server systems
\cite{PR00b,JMM04,MM08,GM08,GS11a,R09} with the most general (single-class) results
appearing in \cite{R09} (and follow-up papers \cite{R07b,PR10}).
As the theory of weak convergence generally relies heavily on the
assumption of compact time intervals, the most general of these results
hold only in the transient regime.
Indeed, with the exception of \cite{HW81} (which treats exponential
processing times), \cite{JMM04} (which treats deterministic
processing times), \cite{GM08} (which treats processing times with
finite support)
and \cite{GS11a} (which treats phase-type processing times and allows
for abandonments and multi-class structure), all of the aforementioned
results are for the associated sequence of normalized \textit{transient}
queue length distributions only, leaving many open questions about the
associated \textit{steady-state} queue length distributions.

In particular, in \cite{GM08} it is shown for the case of
processing times with finite support that the sequence of steady-state
queue length distributions (normalized by $n^{{1/2}}$) is tight,
and has a limit whose tail decays exponentially fast. The authors
further prove that this exponential rate of decay (i.e., large
deviation exponent) is $- 2B(c^2_A + c^2_S)^{-1}$, where $B$ is the
spare capacity parameter, and $c^2_A,c^2_S$ are the squared
coefficients of variation of the inter-arrival and processing time
distributions. In \cite{GM08} it was conjectured that this result
should hold for more general processing time distributions. However,
prior to this work no further progress on this question has been achieved.

In this paper we resolve the conjectures made in \cite
{GM08} with regards to (w.r.t.) tightness of the steady-state queue
length, and take a large step toward resolving the conjectures made
w.r.t. the large deviation exponent. We prove that as long as the
inter-arrival and processing time distributions satisfy minor technical
conditions (e.g., finite $2 + \varepsilon$ moments), the associated
sequence of steady-state queue length distributions, normalized by
$n^{{1/2}}$, is tight.
Under the same minor technical conditions we derive an upper bound on
the large deviation exponent of the limiting steady-state queue length
matching that conjectured by Gamarnik and Momcilovic in \cite{GM08}.
We also prove a matching lower bound when the arrival process is Poisson.

Our main proof technique is the derivation of new and simple
bounds for the FCFS $\mathit{GI}/\mathit{GI}/n$ queue. Our bounds are of a structural
nature, hold for all $n$ and all times $t \geq0$, and have intuitive,
closed-form representations as the suprema of certain natural processes
which converge weakly to Gaussian processes.
Our upper and lower bounds also exhibit a certain duality relationship
and exemplify a general methodology which may be useful for analyzing a
variety of queueing systems. We further illustrate the utility of this
methodology by deriving the first nontrivial bounds for the weak limit
process studied in \cite{R09}.

We note that our techniques allow us to analyze many
properties of the $\mathit{GI}/\mathit{GI}/n$ queue in the H--W regime without having to
consider the complicated exact dynamics of the $\mathit{GI}/\mathit{GI}/n$ queue.
Interestingly, such ideas were used in the original paper of Halfin and
Whitt \cite{HW81} to show tightness of the steady-state queue length
for the $\mathit{GI}/M/n$ queue under the H--W scaling, but do not seem to have
been used in subsequent works on queues in the H--W regime.

The rest of the paper proceeds as follows. In Section \ref
{mainsec}, we present our main results. In Section \ref{uppersec}, we
establish our general-purpose upper bounds for the queue length in a
properly initialized FCFS $\mathit{GI}/\mathit{GI}/n$ queue. In Section~\ref{lowersec},
we establish our general-purpose lower bounds for the queue length in a
properly initialized FCFS $M/\mathit{GI}/n$ queue. In Section \ref{tightsecc}
we use our bounds to prove the tightness of the steady-state queue length
when the system is in the H--W regime. In Section~\ref{ldsecc} we
combine our bounds with known results about weak limits and the suprema
of Gaussian processes to prove our large deviation results. In Section
\ref{appssec} we use our bounds to study the weak limit derived in
\cite{R09}. In Section \ref{concsec} we summarize our main results
and comment on directions for future research. We include a technical
\hyperref[app]{Appendix}.

%s2 #&#
\section{Main results}\label{mainsec}
We consider the first-come-first-serve (FCFS) $\mathit{GI}/\mathit{GI}/n$ queueing model,
in which inter-arrival times are independent and identically
distributed (i.i.d.) random variables (r.v.s), and processing times are
i.i.d. r.v.s.

Let $A$ and $S$ denote some fixed r.v.s with nonnegative
support such that (s.t.) $\E[A] = \mu_A^{-1} < \infty, \E[S] = \mu
_S^{-1} < \infty$ and $\pr(A = 0) = \pr(S = 0) = 0$. Let $\sigma
^2_A$ and $\sigma^2_S$ denote the variance of $A$ and $S$,
respectively. Let $c^2_A$ and $c^2_S$ denote the squared coefficient of
variation (s.c.v.) of $A$ and $S$, respectively.

We fix some excess parameter $B > 0$, and let $\lambda_n
\stackrel{\Delta}{=} n - B n^{{1/2}}$.
For $n$ sufficiently large to ensure $\lambda_n > 0$ (which is assumed
throughout), let $Q^n(t)$ denote the number in system (number in
service${}+{}$number waiting in queue) at time $t$ in the FCFS $\mathit{GI}/\mathit{GI}/n$
queue with inter-arrival times drawn i.i.d. distributed as $A\lambda
_n^{-1}$ and processing times drawn i.i.d. distributed as $S$ (initial
conditions will be specified later), independently from the arrival
process. Note that this scaling is analogous to that studied by Halfin
and Whitt in \cite{HW81}, as the traffic intensity in the $n$th
system is $1 - B n^{-{1/2}}$ in both settings. All processes
should be assumed right-continuous with left limits (r.c.l.l.) unless
stated otherwise. All empty summations should be evaluated as zero, and
all empty products should be evaluated as one.

%s2.1 #&#
\subsection{Main results}
Our main results will require two additional sets of assumptions on $A$
and $S$. The first set of assumptions, which we call the H--W
assumptions, ensures that $\lbrace Q^n(t), n \geq1 \rbrace$ is in the
H--W scaling regime as $n \rightarrow\infty$. We say that $A$ and $S$
satisfy the H--W assumptions if and only if $\mu_A = \mu_S$, in which
case we denote this common rate by $\mu$.
The second set of assumptions, which we call the $T_0$ assumptions, is
a set of additional technical conditions we require for our main results.
\renewcommand\thelonglist{(\roman{longlist})}
\renewcommand\labellonglist{\thelonglist}
\begin{longlist}
\item\label{t1} There exists $\varepsilon> 0$ s.t. $\E[A^{2+\varepsilon}], \E
[S^{2+\varepsilon}] < \infty$.
\item\label{t2} $c^2_A + c^2_S > 0$; namely, either $A$ or $S$ is a nontrivial
r.v.
\item\label{t3} $\limsup_{t \downarrow0} t^{-1} \pr( S \leq t ) < \infty
$.
\item\label{t4} For all sufficiently large $n$ and all initial conditions,
$Q^n(t)$ converges weakly to a stationary measure $Q^n(\infty)$ as $t
\rightarrow\infty$, independent of initial conditions.
\end{longlist}
We now briefly discuss the various assumptions, commenting on both the
reason for their inclusion and their restrictiveness. Condition
\ref{t1} is necessary for several bounds from the literature relating to
suprema of random walks; see \cite{Stight99}. Although we use this
condition to prove tightness of the queue length in the H--W regime, all
our intermediate results about weak limits and Gassian processes would
also hold under only a second moment assumption (as opposed to $2 +
\varepsilon$).

Condition \ref{t3} is necessary for several results from
the literature relating to the weak convergence of scaled renewal
processes; see \cite{Whitt85,W02}.
The condition is (e.g.) satisfied by any discrete distribution with no
mass at zero, any continuous distribution with finite density at zero,
and (more generally) any distribution function (d.f.) which is
absolutely continuous in a neighborhood of zero (see the discussion in
\cite{W02}). All our results other than those pertaining to the weak
convergence of scaled renewal processes and/or the large deviation
exponent of the queue length in the H--W regime would also hold without
this assumption.

Condition \ref{t4} is needed to sensibly discuss the
relevant stationary measures. We refer the interested reader to \cite
{A03} for an excellent discussion of sufficient conditions on $A$ and
$S$ which ensure that \ref{t4} holds, for example, if the d.f. of $A$
is continuous, or more generally has a ``spread-out component''; see
\cite{A03} for details. We note that our nonasymptotic transient
bounds hold even without this condition.

We now state our main results. We begin by establishing the
tightness of the steady-state queue length for the FCFS $\mathit{GI}/\mathit{GI}/n$ queue
in the H--W regime.
%
%th1 #&#
\begin{theorem}\label{tightc}
If $A$ and $S$ satisfy the H--W and $T_0$ assumptions, then the sequence
$\lbrace( Q^n(\infty) - n )^+n^{-{1/2}}, n \geq1
\rbrace$ is tight.
\end{theorem}
In words, the queue length $ ( Q^n(\infty) - n )^+$ scales
like $O(n^{{1/2}})$. Although we conjecture that the sequence
$\lbrace( Q^n(\infty) - n )^+n^{-{1/2}}, n \geq1
\rbrace$ has a unique weak limit (and thus converges weakly), our
approach, which proves the weak convergence of certain bounding
processes for the $\mathit{GI}/\mathit{GI}/n$ queue (but not the $\mathit{GI}/\mathit{GI}/n$ queue itself)
is unable to establish this, and we leave the question of uniqueness as
an interesting open problem.

We now establish an upper bound for the large deviation exponent of
the limiting steady-state queue length for the FCFS $\mathit{GI}/\mathit{GI}/n$ queue in
the H--W regime, and a matching lower bound when the arrival process is Poisson.

%th2 #&#
\begin{theorem}\label{ld1}
Under the same assumptions as Theorem \ref{tightc},
\[
\limsup_{x \rightarrow\infty} x^{-1} \log\Bigl( \limsup
_{n
\rightarrow\infty} \pr\bigl( \bigl( Q^n(\infty) - n
\bigr)^+n^{-{1/2}} > x \bigr) \Bigr) \leq-2 B \bigl(c^2_A
+ c^2_S\bigr)^{-1}.
\]
If in addition $A$ is an exponentially distributed r.v., namely the
system is $M/\mathit{GI}/n$, then
\begin{eqnarray*}
&& \lim_{x \rightarrow\infty} x^{-1} \log\Bigl( \liminf
_{n
\rightarrow\infty} \pr\bigl( \bigl( Q^n(\infty) - n
\bigr)^+n^{-{1/2}} > x \bigr) \Bigr)
\\
&&\qquad= \lim_{x \rightarrow\infty} x^{-1} \log\Bigl( \limsup
_{n
\rightarrow\infty} \pr\bigl( \bigl( Q^n(\infty) - n
\bigr)^+n^{-{1/2}} > x \bigr) \Bigr) = -2 B \bigl(c^2_A
+ c^2_S\bigr)^{-1}.
\end{eqnarray*}
\end{theorem}
In words, Theorem \ref{ld1} states that the tail of the limiting
steady-state queue length is bounded from above by $\exp( - 2 B
(c^2_A + c^2_S)^{-1} x + o(x) )$; and when the arrival process is
Poisson, the tail of the limiting steady-state queue length is bounded
from below by $\exp( - 2 B (c^2_A + c^2_S)^{-1} x - o(x) )$,
where $o(x)$ is some nonnegative function s.t. $\lim_{x \rightarrow
\infty} x^{-1} o(x) = 0$. Theorem \ref{ld1} translates into bounds
for the large deviation behavior of any weak limit of the sequence
$\lbrace(\cal{Q}^n(\infty) - n )^+n^{-{1/2}},
n \geq1 \rbrace$, where at least one weak limit exists by Theorem
\ref{tightc}.

Note that the functional form of the exponent $-2 B (c^2_A +
c^2_S)^{-1}$ shows that the probability of large deviations is a
decreasing function of the\vspace*{1pt} excess parameter $B$, and an increasing
function of the squared coefficients of variation $c^2_A$, $c^2_S$.
This is consistent at an intuitive level, since as $B$ grows, the
system becomes less loaded, which should decrease the the probability
of large
deviations. Similarly, as $c^2_A$ and $c^2_S$ grow, the system becomes
more variable, which should increase the probability of large deviations.

Although we conjecture that $-2 B (c^2_A + c^2_S)^{-1}$
should also be the correct large deviations exponent when $A$ is
non-Markovian, our lower-bounding proof technique relies on certain
properties of the steady-state $M/\mathit{GI}/\infty$ queue which do not hold
for the steady-state $\mathit{GI}/\mathit{GI}/\infty$ queue, and thus we leave such an
extension as an open problem.

%s3 #&#
\section{Upper bound}\label{uppersec}
In this section, we prove a general upper bound for the FCFS $\mathit{GI}/\mathit{GI}/n$
queue, when properly initialized. The bound is valid for all finite
$n$, and works in both the transient and steady-state (when it exists)
regimes. Although we will later customize this bound to the H--W regime
to prove our main results, we note that the bound is in no way limited
to that regime. For a nonnegative r.v. $X$ with finite mean $\E[X] >
0$, let $R(X)$ denote a r.v. distributed as
the residual life distribution of $X$. Namely, for all $z \geq0$,
%
%e1 #&#
\begin{equation}
\label{rezz} \pr\bigl( R(X) > z \bigr) = \bigl(\E[X]\bigr)^{-1} \int
_{z}^{\infty} \pr(X > y) \,dy.
\end{equation}
Recall that associated with a r.v. $X$, an equilibrium renewal process
with renewal distribution $X$ is a counting process in which the first
inter-event time is distributed as $R(X)$, and all subsequent
inter-event times are drawn i.i.d. distributed as $X$; an ordinary
renewal process with renewal distribution $X$ is a counting process in
which all inter-event times, including the first, are drawn i.i.d.
distributed as $X$. Let $\lbrace N_i(t), i=1,\ldots,n \rbrace$ denote
a set of $n$ i.i.d. equilibrium renewal processes with renewal
distribution $S$. Let $A(t)$ denote an equilibrium renewal process with
renewal distribution $A$, with $A(t), \lbrace N_i(t) \rbrace$ mutually
independent.

Let ${\mathcal Q}$ denote the FCFS $\mathit{GI}/\mathit{GI}/n$ queue with
inter-arrival times drawn i.i.d. distributed as $A$, processing times
drawn i.i.d. distributed as $S$, and the following initial conditions.
For $i=1,\ldots,n$, there is a single job initially being processed on
server $i$, and the set of initial processing times of these $n$
initial jobs is drawn i.i.d. distributed as $R(S)$. There are zero jobs
waiting in queue, and the first inter-arrival time is distributed as
$R(A)$, independent of the initial processing times of those jobs
initially in system. We now establish an upper bound for $Q(t)$, the
number in system at time $t$ in ${\mathcal Q}$.
%
%th3 #&#
\begin{theorem}\label{ubound1}
For all $x > 0$, and $t \geq0$,
\[
\pr\bigl( \bigl( Q(t) - n \bigr)^+ > x \bigr) \leq\pr\Biggl( \sup
_{0 \leq s
\leq t} \Biggl( A(s) - \sum_{i=1}^n
N_i(s) \Biggr) > x \Biggr).
\]
If in addition $Q(t)$ converges weakly to a stationary distribution
$Q(\infty)$ as $t \rightarrow\infty$, then for all $x > 0$,
\[
\pr\bigl( \bigl( Q(\infty) - n \bigr)^+ > x \bigr) \leq\pr\Biggl( \sup
_{t
\geq0} \Biggl( A(t) - \sum_{i=1}^n
N_i(t) \Biggr) > x \Biggr).
\]
\end{theorem}

Note that our bounds are monotone in time, as when $t$ increases the
supremum appearing in Theorem \ref{ubound1} is taken over a larger
time window, and the bound for the steady-state is the natural limit of
these transient bounds.\looseness=1

We will prove Theorem \ref{ubound1} by analyzing a different queueing
system $\tilde{\mathcal Q}$ which represents a ``modified'' queue, in
which all servers are kept busy at all times by adding artificial
arrivals whenever a server would otherwise go idle. We note that our
construction is similar to several constructions appearing in the
literature. Our bounding system is closely related to the so-called
\textit{queue with autonomous service}, a model studied previously by
several authors \cite{B65,W70,W02,KS06}, whose dynamics can be
described as the solution to an appropriate Skorokhod problem
\cite{W02,Sk61}. Another related work is \cite{CTK94}, in which the
queue length of the $G/\mathit{GI}/1$ queue is bounded by considering a modified
system in which the server goes on a vacation whenever it would have
otherwise gone idle. Also, in \cite{HW81}, the queue length of the
$\mathit{GI}/M/n$ queue is bounded by considering a modified system in which a
reflecting barrier is placed at state $n$.

We now construct the FCFS $G/\mathit{GI}/n$ queue $\tilde{\mathcal
Q}$ on the same probability space as $\lbrace N_i(t), i=1,\ldots,n
\rbrace$ and $A(t)$.
We begin by defining two auxiliary processes $\tilde{A}(t)$ and
$\tilde{Q}(t)$, where $\tilde{A}(t)$ will become the arrival process
to $\tilde{\mathcal Q}$, and we will later prove that $\tilde{Q}(t)$
equals the number in system in $\tilde{\mathcal Q}$ at time $t$. Let
$\tau_0 \stackrel{\Delta}{=} 0$, $\lbrace\tau_k, k \geq1 \rbrace
$ denote the sequence of event times in the pooled renewal process
$A(t) + \sum_{i=1}^n N_i(t)$, $dA(t) \stackrel{\Delta}{=} A(t) -
A(t^-)$, $A(s,t) \stackrel{\Delta}{=} A(t) - A(s)$ and $dN_i(t)
\stackrel{\Delta}{=} N_i(t) - N_i(t^-)$, $N_i(s,t) \stackrel{\Delta
}{=} N_i(t) - N_i(s)$ for $i=1,\ldots,n$.\looseness=1

We now define the processes $\tilde{A}(t)$ and $\tilde
{Q}(t)$ inductively over $\lbrace\tau_k, k \geq0 \rbrace$. Let
$\tilde{A}(\tau_0) \stackrel{\Delta}{=} 0$, $\tilde{Q}(\tau_0)
\stackrel{\Delta}{=} n$. Now suppose that for some $k \geq0$, we
have defined $\tilde{A}(t)$ and $\tilde{Q}(t)$ for all $t \leq\tau
_k$. We now define these processes for $t \in(\tau_k, \tau_{k+1}]$.
For $t \in(\tau_k, \tau_{k+1})$, let $\tilde{A}(t) \stackrel
{\Delta}{=} \tilde{A}(\tau_k)$, and $\tilde{Q}(t) \stackrel{\Delta
}{=} \tilde{Q}(\tau_k)$. Note that w.p.1 $dA(\tau_{k+1}) + \sum_{i=1}^n
dN_i(\tau_{k+1}) = 1$, since $R(A)$ and $R(S)$ are continuous
r.v.s, $\pr(A = 0) = \pr(S = 0) = 0$, and $A(t), \lbrace N_i(t),
i=1,\ldots,n \rbrace$ are mutually independent. We define
\[
\tilde{A}(\tau_{k+1}) \stackrel{\Delta} {=} \cases{\displaystyle  \tilde{A}(\tau
_k) + 1, & if $dA(\tau_{k+1}) = 1$;
\vspace*{2pt}\cr
\tilde{A}(
\tau_k) + 1, & if $\displaystyle \sum_{i=1}^n
dN_i(\tau_{k+1}) = 1$ and $\displaystyle \tilde{Q}(\tau_{k})
\leq n$;
\vspace*{2pt}\cr
\displaystyle \tilde{A}(\tau_k), & otherwise $\displaystyle \Biggl(\mbox{i.e. }\sum
_{i=1}^n dN_i(\tau_{k+1}) = 1$
and $\displaystyle \tilde{Q}(\tau_{k}) > n\Biggr)$. }
\]
Similarly, we define
\[
\displaystyle \tilde{Q}(\tau_{k+1}) \stackrel{\Delta} {=} \cases{\tilde{Q}(\tau
_k) + 1, & if $\displaystyle dA(\tau_{k+1}) = 1$;
\vspace*{2pt}\cr
\displaystyle \tilde{Q}(
\tau_k), & if $\displaystyle \sum_{i=1}^n
dN_i(\tau_{k+1}) = 1$ and $\displaystyle \tilde{Q}(\tau_{k})
\leq n$;
\vspace*{2pt}\cr
\displaystyle \tilde{Q}(\tau_k) - 1, & otherwise $\displaystyle \Biggl(\mbox{i.e. }\sum
_{i=1}^n dN_i(\tau
_{k+1}) = 1$ and $\displaystyle \tilde{Q}(\tau_{k}) > n\Biggr)$.}
\]
Combining the above completes our inductive definition of $\tilde
{A}(t)$ and $\tilde{Q}(t)$. Since w.p.1 $\lim_{k \rightarrow\infty}
\tau_k = \infty$, it follows that w.p.1 both $\tilde{A}(t)$ and
$\tilde{Q}(t)$ are well defined on $[0,\infty)$. We note that it also
follows from our construction that w.p.1
both $\tilde{A}(t)$ and $\tilde{Q}(t)$ are r.c.l.l., and define
$d\tilde{A}(t) \stackrel{\Delta}{=} \tilde{A}(t) - \tilde{A}(t^-)$.

We now construct the FCFS $G/\mathit{GI}/n$ queue $\tilde{\mathcal
Q}$ using the auxiliary process $\tilde{A}(t)$. Let $V^j_{i}$ denote
the length of the $j$th renewal interval in process $N_i(t), j \geq1,
i = 1,\ldots,n$. Then $\tilde{\mathcal Q}$ is defined to be the FCFS
$G/\mathit{GI}/n$ queue with arrival process $\tilde{A}(t)$ and processing time
distribution $S$, where
the $j$th job assigned to server $i$ (after time 0) is assigned
processing time $V^{j+1}_{i}$ for $j \geq1, i = 1,\ldots,n$. The
initial conditions for $\tilde{\mathcal Q}$
are s.t. for $i = 1,\ldots,n$, there is a single job initially being
processed on server $i$ with initial processing time $V^1_i$, and there
are zero jobs waiting in queue.

We now analyze $\tilde{\mathcal Q}$, proving the following:
%
%le1 #&#
\begin{lemma}\label{busybeaver1}
For $i=1,\ldots,n$, exactly one job departs from server $i$ at each
time $t \in\lbrace\sum_{l=1}^j V^l_i, j \geq1 \rbrace$, and there
are no other departures from server $i$. Also, no server ever idles in
$\tilde{\mathcal Q}$, $\tilde{Q}(t)$ equals the number in system in
$\tilde{\mathcal Q}$ at time $t$ for all $t \geq0$, and for all $k
\geq1$,
%
%e2 #&#
\begin{equation}
\label{recursenumber} \tilde{Q}( \tau_{k} ) - n = \max\Biggl( 0,
\tilde{Q}( \tau_{k-1} ) - n + dA ( \tau_{k} ) - \sum
_{i=1}^n dN_i ( \tau_{k} )
\Biggr).
\end{equation}
\end{lemma}
\begin{pf} The proof proceeds by induction on $\lbrace\tau_k, k \geq
0 \rbrace$, with induction hypothesis that the lemma holds for all $t
\leq\tau_k$.
The base case $k = 0$ follows from the the initial conditions of
$\tilde{\mathcal Q}$ and $\tilde{Q}(t)$.
Thus assume that the induction hypothesis holds for some fixed $k \geq
0$. We first establish the induction step for the statements about the
departure process and nonidling of servers. Let us fix some $i \in
\lbrace1,\ldots,n \rbrace$. By the induction hypothesis, server $i$
was nonidling on $[0,\tau_k]$, and the set of departure times from
server $i$ on $[0,\tau_k]$ was exactly $\lbrace\sum_{l=1}^j V^l_i, j
= 1,\ldots,N_i(\tau_k) \rbrace$. We claim that the next departure
from server $i$ occurs at time $\sum_{l=1}^{N_i(\tau_k)+1} V^l_i$.
Indeed, if $N_i(\tau_k) = 0$, the next departure from server $i$ is
the first departure from server~$i$, which occurs at time $V^1_i$. If
instead $N_i(\tau_k) > 0$, then the last departure from server $i$ to
occur at or before time $\tau_k$ occurred at time $\sum_{l=1}^{N_i(\tau
_k)} V^l_i$. At that time a new job began processing
on server $i$ with processing time $V^{N_i(\tau_k)+1}_i$. This job
will depart at time $\sum_{l=1}^{N_i(\tau_k)+1} V^l_i$, verifying the
claim. It follows that no server idles on $(\tau_k,\tau_{k+1})$,
since $\sum_{l=1}^{N_i(\tau_k)+1} V^l_i \in\lbrace\tau_j, j \geq1
\rbrace$, and thus $\sum_{l=1}^{N_i(\tau_k)+1} V^l_i \geq\tau
_{k+1}$. We now treat two cases. First, suppose $\sum_{l=1}^{N_i(\tau
_k)+1} V^l_i > \tau_{k+1}$. Then there are no departures from server
$i$ on $(\tau_k, \tau_{k+1}]$ and the induction step follows
immediately from the induction hypothesis. Alternatively, suppose $\sum
_{l=1}^{N_i(\tau_k)+1} V^l_i = \tau_{k+1}$. In this case the next
departure from server $i$ occurs at time $\tau_{k+1}$, $dN_i(\tau
_{k+1}) = 1$, and all other servers are nonidling and have no
departures on $(\tau_k,\tau_{k+1}]$. Thus if there are at least $n+1$
jobs in $\tilde{\mathcal Q}$ at time $\tau_k$, then there are at
least $n+1$ jobs in $\tilde{\mathcal Q}$ at time $\tau_{k+1}^-$, and
some job begins processing on server $i$ at time $\tau_{k+1}$.
Alternatively, if there are exactly $n$ jobs in $\tilde{\mathcal Q}$
at time $\tau_k$,
then $\tilde{Q}(\tau_k) = n$ by the induction hypothesis. Thus
$d\tilde{A}(\tau_{k+1}) = 1$, and this arrival immediately begins
processing on server $i$. Combining the above treats all cases since
there are at least $n$ jobs in $\tilde{\mathcal Q}$ at time $\tau_k$
by the induction hypothesis, completing the induction step.

We now prove the induction step for the statement that
$\tilde{Q}(t)$ equals the number in system in $\tilde{\mathcal Q}$ at
time $t$, as well as (\ref{recursenumber}). Since we have already
proven that any departures from $\tilde{\mathcal Q}$ on $(\tau_k,
\tau_{k+1}]$ occur at time $\tau_{k+1}$, and by construction any
jumps in $\tilde{A}(t)$ and $\tilde{Q}(t)$ on $(\tau_k, \tau
_{k+1}]$ occur at time $\tau_{k+1}$, it suffices to prove that $\tilde
{Q}(\tau_{k+1})$ equals the number in system in $\tilde{\mathcal Q}$
at time $\tau_{k+1}$. First, suppose $dA(\tau_{k+1}) = 1$. Then $\sum
_{i=1}^n dN_i( \tau_{k+1}) = 0$, $\tilde{Q}( \tau_{k} ) \geq n$
by the induction hypothesis, and $\tilde{Q}( \tau_{k+1} ) = \tilde
{Q}( \tau_{k} ) + 1$. Thus
\begin{eqnarray*}
\max\Biggl(0, \tilde{Q}( \tau_{k} ) - n + dA( \tau_{k+1} )
- \sum_{i=1}^{n} dN_i(
\tau_{k+1} ) \Biggr) &=& \max\bigl(0, \tilde{Q}( \tau_{k} ) -
n + 1 \bigr)
\\
&=& \tilde{Q}( \tau_{k} ) - n + 1
\\
&=& \tilde{Q}( \tau_{k+1} ) - n,
\end{eqnarray*}
showing that (\ref{recursenumber}) holds. Note that $\sum_{i=1}^n
dN_i( \tau_{k+1}) = 0$ implies that\break $\sum_{l=1}^{N_i(\tau_k)+1}
V^l_i > \tau_{k+1}$ for all $i = 1,\ldots,n$, and we have already
proven that in this case there are no departures from $\tilde{\mathcal
Q}$ on $(\tau_k, \tau_{k+1}]$. Since $dA(\tau_{k+1}) = 1$ implies
$d\tilde{A}(\tau_{k+1}) = 1$, it follows that the number in system in
$\tilde{\mathcal Q}$ at time $\tau_{k+1}$ is one more than the number
in system in $\tilde{\mathcal Q}$ at time $\tau_{k}$. Thus $\tilde
{Q}(\tau_{k+1})$ equals the number in system in $\tilde{\mathcal Q}$
at time $\tau_{k+1}$ by the induction hypothesis.

Now suppose that $\sum_{i=1}^n dN_i( \tau_{k+1} ) = 1$.
Then $dA(\tau_{k+1}) = 0$, and there exists a unique index $i^*$ s.t.
$\sum_{l=1}^{N_{i^*}(\tau_k)+1} V^l_{i^*} = \tau_{k+1}$. We have
already proven that in this case there are no departures from
$\tilde{\mathcal Q}$ on $(\tau_k,\tau_{k+1})$, and a single
departure from $\tilde{\mathcal Q}$ at time $\tau_{k+1}$ (on server
$i^*$). First suppose that
there are at least $n+1$ jobs in $\tilde{\mathcal Q}$ at time~$\tau
_k$. Then $\tilde{Q}( \tau_{k} ) \geq n + 1$ by the induction
hypothesis, and $\tilde{Q}( \tau_{k+1} ) = \tilde{Q}( \tau_{k} ) -
1$. Thus
\begin{eqnarray*}
\max\Biggl(0, \tilde{Q}( \tau_{k} ) - n + dA( \tau_{k+1} )
- \sum_{i=1}^n dN_i(
\tau_{k+1} ) \Biggr) &=& \max\bigl(0, \tilde{Q}( \tau_{k} ) -
n - 1 \bigr)
\\
&=& \tilde{Q}( \tau_{k} ) - n - 1
\\
&=& \tilde{Q}( \tau_{k+1} ) - n,
\end{eqnarray*}
showing that (\ref{recursenumber}) holds. Since $d\tilde{A}(\tau
_{k+1}) = 0$, there are no arrivals to $\tilde{\mathcal Q}$ on $(\tau
_k, \tau_{k+1}]$. Combining the above, we find that the number in
system in $\tilde{\mathcal Q}$ at time $\tau_{k+1}$ is one less than
the number in system in $\tilde{\mathcal Q}$ at time $\tau_{k}$. Thus
$\tilde{Q}(\tau_{k+1})$ equals the number in system in $\tilde
{\mathcal Q}$ at time $\tau_{k+1}$ by the induction hypothesis.

Alternatively, suppose that $\sum_{i=1}^n dN_i( \tau_{k+1}
) = 1$ and there are exactly $n$ jobs in $\tilde{\mathcal Q}$ at time
$\tau_k$.
Then $\tilde{Q}( \tau_{k} ) = n$ by the induction hypothesis, and
$\tilde{Q}( \tau_{k+1} ) = \tilde{Q}( \tau_{k} )$. Thus
\begin{eqnarray*}
\max\Biggl(0, \tilde{Q}( \tau_{k} ) - n + dA( \tau_{k+1} )
- \sum_{i=1}^n dN_i(
\tau_{k+1} ) \Biggr) &=& \max\bigl(0, \tilde{Q}( \tau_{k} ) -
n - 1 \bigr)
\\
&=& 0
\\
&=& \tilde{Q}( \tau_{k+1} ) - n,
\end{eqnarray*}
showing that (\ref{recursenumber}) holds. Since $d\tilde{A}(\tau
_{k+1}) = 1$, there is a single arrival to $\tilde{\mathcal Q}$ on
$(\tau_k, \tau_{k+1}]$. Combining the above, we find that the number
in system in $\tilde{\mathcal Q}$ at time $\tau_{k+1}$ equals the
number in system in $\tilde{\mathcal Q}$ at time $\tau_{k}$. Thus
$\tilde{Q}(\tau_{k+1})$ equals the number in system in $\tilde
{\mathcal Q}$ at time $\tau_{k+1}$ by the induction hypothesis. Since
$\tilde{Q}(\tau_k) \geq n$ by the induction hypothesis, this treats
all cases, completing the proof of the induction and the lemma.
\end{pf}
We now ``unfold'' recursion (\ref{recursenumber}) to derive a simple
one-dimensional random walk representation for $\tilde{Q}(t)$. The
relationship between recursions such as (\ref{recursenumber}) and the
suprema of associated one-dimensional random walks is well known (see
\cite{B65,CTK94}), and can also be formalized by studying
the appropriate Skorokhod problem \cite{Sk61}.
Furthermore, although it seems that $\tilde{Q}(t) - n$ cannot be
immediately related to the Skorokhod problem naturally associated with
$Q(t) - n$, we note that such a formulation may be possible through the
framework of jump reflection; see \cite{Prot80}.

Then it follows from (\ref{recursenumber}) and a
straightforward induction on $\lbrace\tau_k, k \geq0 \rbrace$ that
w.p.1, for all $k \geq0$,
\[
\tilde{Q}( \tau_k ) - n = \max_{0 \leq j \leq k} \Biggl( A (
\tau_{k-j},\tau_{k} ) - \sum_{i=1}^n
N_i ( \tau_{k-j}, \tau_k ) \Biggr).
\]
As all jumps in $\tilde{Q}(t)$ occur at times $t \in\lbrace\tau_k,
k \geq1 \rbrace$, we have the following:\eject
%
%co1 #&#
\begin{corollary}\label{uprop10}
W.p.1, for all $t \geq0$,
\[
\tilde{Q}(t) - n = \sup_{0 \leq s \leq t} \Biggl( A ( t - s,t ) - \sum
_{i=1}^n N_i ( t - s, t )
\Biggr).
\]
\end{corollary}

We now prove that $\tilde{Q}(t)$ provides an upper bound for $Q(t)$.
%
%pr1 #&#
\begin{proposition}\label{uprop2}
$Q(t)$ and $\tilde{Q}(t)$ can be constructed on the same probability
space so that w.p.1 $Q(t) \leq\tilde{Q}(t)$ for all $t \geq0$.
\end{proposition}

For our later results, it will be useful to first prove a general
comparison result for $G/G/n$ queues. Although such results seem to be
generally known in the queueing literature (see \cite{W81c,SY89}), we
include a proof for completeness. For an event~$E$, let $I(E)$ denote
the indicator function of $E$.
%
%le2 #&#
\begin{lemma}\label{qom}
Let ${\mathcal Q}^1$ and ${\mathcal Q}^2$ be two FCFS $G/G/n$ queues
with finite, strictly positive inter-arrival and processing times.
Let $\lbrace T^i_k, k \geq1 \rbrace$ denote the ordered sequence of
arrival times to ${\mathcal Q}^i$, $i \in\lbrace1,2 \rbrace$. Let
$S^i_k$ denote the processing time assigned to the job that arrives to
${\mathcal Q}^i$ at time $T^i_k$, $k \geq1, i \in\lbrace1,2 \rbrace$.
Further suppose that:
\begin{longlist}
\item\label{as1} the initial number in system in ${\mathcal Q}^1$ is at most
$n$;
\item\label{as2} for each job $J$ initially in ${\mathcal Q}^1$, there is a
distinct corresponding job $J'$ initially in ${\mathcal Q}^2$
s.t. the initial processing time of $J$ in ${\mathcal Q}^1$ equals the
initial processing time of $J'$ in ${\mathcal Q}^2$;
\item\label{as3}$\lbrace T^1_k, k \geq1\rbrace$ is a subsequence of $\lbrace
T^2_k, k \geq1 \rbrace$;
\item\label{as4} for all $k \geq1$, the job that arrives to ${\mathcal Q}^2$ at
time $T^1_k$ is assigned processing time $S^1_k$, the same processing
time assigned to the job which arrives to ${\mathcal Q}^1$ at that
time.
\end{longlist}

Then the number in system in ${\mathcal Q}^2$ at time $t$ is at least
the number in system in ${\mathcal Q}^1$ at time $t$ for all $t \geq0$.
\end{lemma}
\begin{pf}
The proof is deferred to the \hyperref[app]{Appendix}.
\end{pf}

We now complete the proof of Proposition \ref{uprop2}.
\begin{pf*}{Proof of Proposition \ref{uprop2}}
We construct $\tilde{\mathcal Q}$ and ${\mathcal Q}$ on the same
probability space. We assign ${\mathcal Q}$ and $\tilde{\mathcal Q}$
the same initial conditions, and let $A(t)$ be the arrival process to
${\mathcal Q}$ on $(0,\infty)$. Let $\lbrace t_k, k \geq1 \rbrace$
denote the ordered sequence of event times in $A(t)$. It follows from
the construction of $\tilde{A}(t)$ that $\lbrace t_k, k \geq1
\rbrace$ is a subsequence of the set of event times in $\tilde{A}(t)$.
We let the processing time assigned to the arrival to $\tilde{\mathcal
Q}$ at time $t_k$ equal the processing time assigned to the arrival to
${\mathcal Q}$ at time $t_k$, $k \geq1$. It follows that w.p.1
${\mathcal Q}$ and $\tilde{\mathcal Q}$ satisfy the conditions of
Lemma \ref{qom}. Combining the above with Lemma~\ref{busybeaver1}
completes the proof.\vadjust{\goodbreak}
\end{pf*}

We now complete the proof of Theorem \ref{ubound1}.
\begin{pf*}{Proof of Theorem \ref{ubound1}}
By elementary renewal theory (see \cite{Cox70}), ${A}(s)_{0 \leq s
\leq t}$ has the same distribution (on the process level) as ${A}
(t - s, t )_{0 \leq s \leq t}$, and $\sum_{i=1}^n N_i(s)_{0 \leq s
\leq t}$ has the same distribution (on the process level) as\break
$\sum_{i=1}^n N_i (t - s, t )_{0 \leq s \leq t}$. Combining
with the independence of $A(t)$ and\break $\sum_{i=1}^n N_i(t)$, Corollary
\ref{uprop10} and Proposition \ref{uprop2}, proves the theorem.

We now prove the corresponding steady-state result. Note that
for any $x > 0$, the sequence of events $ \lbrace\sup_{0 \leq s
\leq t} ( A(s) - \sum_{i=1}^n N_i(s) ) > x, t \geq0
\rbrace$ is monotonic in~$t$. It follows from the continuity of
probability measures that
\[
\lim_{t \rightarrow\infty} \pr\Biggl( \sup_{0 \leq s \leq t} \Biggl(
A(s) - \sum_{i=1}^n N_i(s)
\Biggr) > x \Biggr) = \pr\Biggl( \sup_{t
\geq0} \Biggl( A(t) - \sum
_{i=1}^n N_i(t) \Biggr) > x
\Biggr).
\]
The steady-state result then follows from the corresponding transient
result and the definition of weak convergence, since $Q(\infty)$ has
integer support.
\end{pf*}

%s4 #&#
\section{Lower bound}\label{lowersec}
In this section, we prove a general lower bound for the $M/\mathit{GI}/n$ queue,
when properly initialized. Suppose $A$ is an exponentially distributed
r.v. Let $Z$ denote a Poisson r.v. with mean
$\frac{\mu_A}{\mu_S}$.\vspace*{1pt}
Let ${\mathcal Q}_2$ denote the $M/\mathit{GI}/n$ queue with inter-arrival times
drawn i.i.d. distributed as $A$, processing times drawn i.i.d.
distributed as $S$ and the following initial conditions. At time 0
there are $Z$ jobs in system. This set of initial jobs have initial
processing times drawn i.i.d. distributed as $R(S)$, independent of
$Z$. If $Z \geq n$, a set of exactly $n$ initial jobs is selected
uniformly at random (u.a.r.) to be processed initially, and the
remaining initial jobs queue for processing. Suppose also that the
first inter-arrival time is distributed as $R(A)$ (also an
exponentially distributed r.v.) independent of both $Z$ and the initial
processing times of those jobs initially in the system. Recall the
processes $A(t)$ and $\lbrace N_i(t), i=1,\ldots,n \rbrace$, which
were defined previously at the start of Section \ref{uppersec}. Then
$Q_2(t)$, the number in system at time $t$ in ${\mathcal Q}_2$, satisfies
%
%th4 #&#
\begin{theorem}\label{lbound1}
For all $x > 0$, and $t \geq0$,
\[
\pr\bigl( \bigl( Q_2(t) - n \bigr)^+ > x \bigr) \geq\pr( Z \geq n
) \sup_{0 \leq s \leq t} \pr\Biggl( A(s) - \sum
_{i=1}^n N_i(s) > x \Biggr).
\]
If in addition $Q_2(t)$ converges weakly to a stationary distribution
$Q(\infty)$ as \mbox{$t \rightarrow\infty$}, then for all $x > 0$,
\[
\pr\bigl( \bigl( Q(\infty) - n \bigr)^+ > x \bigr) \geq\pr( Z \geq n
) \sup
_{t \geq0} \pr\Biggl( A(t) - \sum_{i=1}^n
N_i(t) > x \Biggr).
\]
\end{theorem}
Comparing with Theorem \ref{ubound1}, we see that our upper and lower
bounds exhibit a certain duality, marked by the order of the $\pr$ and
$\sup$ operators.\eject

We will prove Theorem \ref{lbound1} by coupling ${\mathcal
Q}_2$ to \textit{both} an associated FCFS $M/\mathit{GI}/\infty$ queue
${\mathcal Q}_{\infty}$ and a certain family of FCFS $G/G/n$ queues
$\lbrace{\mathcal Q}^s_2,\break s \geq0 \rbrace$. For each \mbox{$s \geq0$},
our coupling ensures that $Q^s_2(t)$, the number in system at time $t$
in ${\mathcal Q}^s_2$, provides a lower bound for $Q_2(t)$ for all $t
\geq s$, and that the set of remaining processing times (at time $s$)
of those jobs in ${\mathcal Q}^s_2$ at time $s$ is a random thinning of
the set of remaining processing times (at time~$s$) of those jobs in
${\mathcal Q}_{\infty}$ at time $s$. We note that some of the ideas
involved in the proof of our lower bound have appeared in the
literature before; see \cite{Stoyan83,SY89,W00}.

We now construct ${\mathcal Q}_{\infty}$ and $\lbrace
{\mathcal Q}^s_2, s \geq0 \rbrace$.
We assign ${\mathcal Q}_{\infty}$ the same initial conditions as
${\mathcal Q}_2$ (although in ${\mathcal Q}_{\infty}$ all initial jobs
begin processing at time~$0$). We let ${\mathcal Q}_{\infty}$ and
${\mathcal Q}_2$ have the same arrival process, and for each arrival,
we let the processing time assigned to this arrival to ${\mathcal
Q}_{\infty}$ equal the processing time assigned to this arrival to
${\mathcal Q}_2$.

We now describe the initial conditions and arrival process
for ${\mathcal Q}^s_2$ in terms of an appropriate thinning of the
initial conditions and arrival process of ${\mathcal Q}_{\infty}$,
where the nature of this thinning depends on $Q_{\infty}(s)$, the
number in system at time $s$ in ${\mathcal Q}_{\infty}$.
If $Q_{\infty}(s) < n$, then the initial conditions of ${\mathcal
Q}^s_2$ are to have zero jobs in system, and the arrival process to
${\mathcal Q}^s_2$ is to have zero arrivals on $[0,\infty)$. If
$Q_{\infty}(s) \geq n$, then we select a size-$n$ subset ${\mathcal
C}^s$ of jobs u.a.r. from all subsets of the jobs being processed in
${\mathcal Q}_{\infty}$ at time $s$. Let ${\mathcal C}^s_0$ denote
those jobs in ${\mathcal C}^s$ which were initially in ${\mathcal
Q}_{\infty}$ at time~$0$.
Then the initial conditions of ${\mathcal Q}^s_2$ are as follows. For
each job $J \in{\mathcal C}^s_0$, there is a corresponding job $J'$
initially in ${\mathcal Q}^s_2$, where the initial processing time of
$J'$ in ${\mathcal Q}^s_2$ equals the initial processing time of $J$ in
${\mathcal Q}_{\infty}$. There are no other initial jobs in~${\mathcal
Q}^s_2$. The arrival process to ${\mathcal Q}^s_2$ on $(0,s]$ is as
follows. For each job $J$ that arrives to ${\mathcal Q}_{\infty}$ (and
thus to ${\mathcal Q}_2$) on $(0,s]$, say at time $\tau$, there is a
corresponding arrival $J'$ to ${\mathcal Q}^s_2$ at time $\tau$ if and
only if $J \in{\mathcal C}^s \setminus{\mathcal C}^s_0$. In this
case, the processing time assigned to $J'$ in ${\mathcal Q}^s_2$ equals
the processing time assigned to $J$ in ${\mathcal Q}_{\infty}$. There
are no other arrivals to ${\mathcal Q}^s_2$ on $(0,s]$. We let
${\mathcal Q}^s_2$, ${\mathcal Q}_{\infty}$ and ${\mathcal Q}_2$ have
the same arrival process on $(s,\infty)$, and for each arrival, we let
the processing time assigned to this arrival to ${\mathcal Q}^s_2$
equal the processing time assigned to this arrival to ${\mathcal
Q}_{\infty}$ (and thus ${\mathcal Q}_2$).

We claim that our coupling of ${\mathcal Q}_{\infty}$ to
${\mathcal Q}_2$ and construction of ${\mathcal Q}^s_2$ ensure that
${\mathcal Q}^s_2$ and ${\mathcal Q}_2$ satisfy the conditions of
Lemma \ref{qom}. Indeed, for each job initially in ${\mathcal
Q}^s_2$, there is a distinct corresponding job initially in ${\mathcal
Q}_2$ with the same initial processing time. Also, for each job that
arrives to ${\mathcal Q}^s_2$, there is a distinct corresponding job
that arrives to ${\mathcal Q}_2$ at the same time with the same
processing time.
Thus w.p.1 $Q^s_2(t)$, the number in system at time $t$ in ${\mathcal
Q}^s_2$, satisfies
%
%e3 #&#
\begin{equation}
\label{qdom1} Q_2(t) \geq Q^s_2(t) \qquad\mbox{for
all } s,t \geq0.
\end{equation}
We now complete the proof of Theorem \ref{lbound1}.\eject
\begin{pf*}{Proof of Theorem \ref{lbound1}}
Since ${\mathcal Q}_{\infty}$ is initialized with its stationary
measure (see \cite{Tak62}), it follows from the basic properties of
the $M/\mathit{GI}/\infty$ queue (see \cite{Tak62}) that $\pr( Q_{\infty
}(s) \geq n) = \pr( Z \geq n )$, and conditionally on the event
$\lbrace Q_{\infty}(s) \geq n \rbrace$, the set of remaining
processing times (at time $s$) of those jobs being processed in
${\mathcal Q}_{\infty}$ at time $s$ are drawn i.i.d. distributed as
$R(S)$. Thus conditionally on the event $\lbrace Q_{\infty}(s) \geq n
\rbrace$, one has that $|{\mathcal C}^s| = n$, and the set of
remaining processing times (at time $s$, in ${\mathcal Q}_{\infty}$)
of those jobs belonging to ${\mathcal C}^s$ is drawn i.i.d. distributed
as $R(S)$.

By construction the number of jobs initially in ${\mathcal
Q}^s_2$ at time $0$ \textit{plus} the number of jobs that arrive to
${\mathcal Q}^s_2$ on $(0,s]$ is at most $n$. Thus all jobs initially
in ${\mathcal Q}^s_2$ at time $0$ and all jobs that arrive to
${\mathcal Q}^s_2$ on $(0,s]$ begin processing immediately in
${\mathcal Q}^s_2$, as if ${\mathcal Q}^s_2$ were an infinite-server
queue. It follows from our construction that conditionally on the event
$\lbrace Q_{\infty}(s) \geq n \rbrace$, the set of remaining
processing times (at time $s$) of the $n$ jobs in ${\mathcal Q}^s_2$ at
time $s$ equals the set of remaining processing times (at time $s$, in
${\mathcal Q}_{\infty}$) of those jobs belonging to ${\mathcal C}^s$,
and are thus drawn i.i.d. distributed as $R(S)$.

Let us fix some $s,t$ s.t. $0 \leq s \leq t$. Recall that
$V^j_{i}$ denotes the length of the $j$th renewal interval in process
$N_i(t), j \geq1, i = 1,\ldots,n$.
It follows from our construction that conditionally on the event
$\lbrace Q_{\infty}(s) \geq n \rbrace$, we may set the remaining
processing time (at time $s$) of the job on server $i$ in ${\mathcal
Q}^s_2$ at time $s$ equal to $V^1_{i}$. We can also set the processing
time of the $j$th job assigned to server $i$ in ${\mathcal Q}^s_2$
(after time $s$) equal to $V^{j+1}_{i}$.
Under this coupling the total number of jobs that depart from server
$i$ in ${\mathcal Q}^s_2$ during $[s,t]$ is at most $N_i(t-s)$, and
therefore the total number of departures from ${\mathcal Q}^s_2$ during
$[s,t]$ is at most $\sum_{i=1}^n N_i(t-s)$, independent of the arrival
process to ${\mathcal Q}^s_2$ on $[s,t]$. By the memoryless and
stationary increments properties of the Poisson process, we may let the
arrival process to ${\mathcal Q}^s_2$ on $[s,t]$ equal $A(v)_{0 \leq v
\leq t-s}$. Combining the above, we find that for all $x > 0$, $\pr(
Q^s_2(t) - n > x) \geq\pr( Z \geq n ) \pr( A(t-s) - \sum_{i=1}^n
N_i(t-s) > x )$. Observing that $s$ was general, we may
then take the supremum of the above bound over all $s \in[0,t]$, and
combine with (\ref{qdom1}) to complete the proof of the theorem. The
corresponding steady-state result then follows from the fact that
monotonic sequences have limits and the definition of weak convergence.
\end{pf*}

%s5 #&#
\section{\texorpdfstring{Tightness and proof of Theorem \protect\ref{tightc}}
{Tightness and proof of Theorem 1}}\label{tightsecc}
In this section, we prove Theorem~\ref{tightc}. We note that it
follows almost immediately from Theorem \ref{ubound1} and well-known
tightness results from the literature (see \cite{B99}, Theorem 14.6,
\cite{W02}, Theorem 7.2.3) that for any \textit{fixed} $T \geq0$,
$\lbrace n^{-{1/2}} (Q^n(t) - n )^+_{0 \leq t \leq T},
n \geq1 \rbrace$ is tight in the space $D[0,T]$ under the $J_1$
topology; see Section \ref{weaksubsec} for details. The challenge
is that when analyzing $\lbrace n^{-{1/2}} (Q^n(\infty
)-n )^+, n \geq1\rbrace$, one does not have the luxury of bounded
time intervals. In particular, to apply Theorem \ref{ubound1}, we
must show tightness of a supremum taken
over an infinite time horizon. For this reason, most standard weak
convergence type results and arguments from the literature (see \cite
{W02}) break down, and cannot immediately be applied. Instead, we will
relate the supremum appearing in the right-hand side (r.h.s.) of
Theorem \ref{ubound1} to the steady-state waiting time in an
appropriate $G/D/1$ queue with stationary (as opposed to i.i.d.)
inter-arrival times. We will then apply known results from the
literature, in particular \cite{Stight99}, to show that under the H--W
scaling this sequence of steady-state waiting times, properly
normalized, is tight.

Suppose that assumptions H--W and $T_0$ hold. Let $A_n(t)
\stackrel{\Delta}{=} A(\lambda_n t)$. In light of Theorem \ref
{ubound1}, it suffices to prove that $\lbrace n^{-{1/2}} \sup_{
t \geq0} (A_n(t) - \sum_{i=1}^n N_i(t) ),\break n \geq1 \rbrace$
is tight.
Let ${\mathsf A^0_n}(t)$ denote an ordinary renewal process with
renewal distribution $A\lambda_n^{-1}$, independent of $\lbrace
N_i(t), i=1,\ldots,n \rbrace$. Note that we may construct $A_n(t)$
and $A^0_n(t)$ on the same probability space so that $A_n(t) \leq1 +
A^0_n(t)$ for all $t \geq0$. It thus suffices to demonstrate the
tightness of $\lbrace n^{-{1/2}} \sup_{ t \geq0}
(A^0_n(t) - \sum_{i=1}^n N_i(t) ),\allowbreak n \geq1 \rbrace$.

Let $\lbrace A^1_i, i \geq1 \rbrace$ denote a countably
infinite sequence of r.v.s drawn i.i.d. distributed as $A$, independent
of $\lbrace N_i(t), i=1,\ldots,n \rbrace$. Note that since $A^0_n(t)
- \sum_{i=1}^n N_i(t)$ only increases at jumps of $A^0_n(t)$, we may
construct $A^0_n(t),\break \sum_{i=1}^n N_i(t)$, and $\lbrace A^1_i, i \geq
1 \rbrace$ on the same probability space so that
%
%e4 #&#
\begin{equation}
\label{tightenit}\quad n^{-{1/2}} \sup_{ t \geq0}
\Biggl(A^0_n(t) - \sum_{i=1}^n
N_i(t) \Biggr)= n^{-{1/2}}\sup_{k \geq0} \Biggl(
k - \sum_{i=1}^n N_i\Biggl(
\lambda_n^{-1} \sum_{j=1}^k
A^1_j \Biggr) \Biggr).
\end{equation}

We now show that
%
%e5 #&#
\begin{equation}
\label{tightenitt2} \Biggl\lbrace n^{-{1/2}}\sup_{k \geq0}
\Biggl( k - \sum_{i=1}^n N_i
\Biggl( \lambda_n^{-1} \sum_{j=1}^k
A^1_j \Biggr) \Biggr), n \geq1 \Biggr\rbrace
\end{equation}
is tight, which (by the above) will imply Theorem \ref{tightc}.
Fortunately, the tightness of such sequences of suprema has already
been addressed in the literature, in the context of steady-state
waiting times in a $G/G/1$ queue, with stationary inter-arrival times,
in heavy-traffic. In particular, note that for $M \geq1$, $\sup_{0
\leq k \leq M} ( k - \sum_{i=1}^n N_i( \lambda_n^{-1} \sum_{j=1}^k
A^1_j ) )$ corresponds to the waiting time of the
$(M+1)$st arrival to a $G/D/1$ queue, initially empty, with all
processing times equal to 1, and the $k$th inter-arrival time equal to
\[
\sum_{i=1}^n N_i \Biggl(
\lambda_n^{-1} \sum_{j=1}^{M-k}
A^1_j, \lambda_n^{-1} \sum
_{j=1}^{M-k+1} A^1_j
\Biggr),\qquad k \leq M.
\]

Recall that $\sum_{i=1}^n N_i(t)_{t \geq0}$ has the same
distribution (on the process level) as $\sum_{i=1}^n N_i
(t-s,t )_{0 \leq s \leq t}$ (see \cite{Cox70}), and $\lbrace
A^1_i, i \geq1 \rbrace$ are i.i.d. It follows that for all $M \geq
1$, $\sup_{0 \leq k \leq M} ( k - \sum_{i=1}^n N_i( \lambda
_n^{-1} \sum_{j=1}^k A^1_j ) )$ also has the same distribution as
the waiting time of the $(M+1)$st arrival to a $G/D/1$ queue, initially
empty, with all processing times equal to 1, and the $k$th
inter-arrival time equal to
\[
\sum_{i=1}^n N_i \Biggl(
\lambda_n^{-1} \sum_{j=1}^{k-1}
A^1_j, \lambda_n^{-1} \sum
_{j=1}^{k} A^1_j
\Biggr),\qquad k \geq1.
\]
For this queueing model, in which the sequence of inter-arrival times
is stationary, one can ask whether there is a meaningful notion of
steady-state waiting time, whose distribution would naturally coincide
with that of
\[
\lim_{M \rightarrow\infty} \sup_{0 \leq k \leq M} \Biggl( k - \sum
_{i=1}^n N_i\Biggl(
\lambda_n^{-1} \sum_{j=1}^k
A^1_j \Biggr) \Biggr) = \sup_{k
\geq0}
\Biggl( k - \sum_{i=1}^n N_i
\Biggl( \lambda_n^{-1} \sum_{j=1}^k
A^1_j \Biggr) \Biggr).
\]
Furthermore, should one examine a sequence of such queues in heavy
traffic, one can ask whether the corresponding sequence of steady-state
waiting times, properly normalized, is tight.

Note that as (\ref{tightenitt2}) is such a sequence, we are
left to answer exactly this question. Fortunately, sufficient
conditions for tightness of such a sequence are given in \cite
{Stight99}. In particular, as we will show, from the results of \cite
{Stight99} (in the notation of \cite{Stight99}), we have the following:
%
%th5 #&#
\begin{theorem}\label{statwait2}
Suppose that for all sufficiently large $n$, $\lbrace\zeta_{n,i}, i
\geq1 \rbrace$ is a stationary, countably infinite sequence of r.v.
Let $a_n \stackrel{\Delta}{=} \E[ \zeta_{n,1} ]$, and $W_{n,k}
\stackrel{\Delta}{=} \sum_{i=1}^k \zeta_{n,i}$. Further assume that
$a_n < 0, \lim_{n \rightarrow\infty} a_n = 0$, and there exist
$C_1,C_2 < \infty$ and $\varepsilon> 0$ s.t. for all sufficiently large $n$:
\begin{longlist}
\item
$\E[ | W_{n,k} - k a_n |^{2 + \varepsilon} ] \leq C_1 k^{1 +
{\varepsilon}/{2}}$ for all $k \geq1$;
\item$\pr( \max_{i=1,\ldots,k} (W_{n,i} - i a_n) > x )
\leq C_2 k^{1 + {\varepsilon}/{2}} x^{- (2 + \varepsilon)}$ for all $k
\geq1$ and \mbox{$x > 0$}.
\end{longlist}
Then $\lbrace|a_n| \sup_{k \geq0} W_{n,k}, n \geq1 \rbrace$ is tight.
\end{theorem}
\begin{pf}
The proof follows from Theorem 1 of \cite{Stight99}, and is deferred
to the \hyperref[app]{Appendix}.
\end{pf}

To verify that the assumptions of Theorem \ref{statwait2} hold for
\[
\Biggl\lbrace n^{-{1/2}}\sup_{k \geq0} \Biggl( k - \sum
_{i=1}^n N_i\Biggl(
\lambda_n^{-1} \sum_{j=1}^k
A^1_j \Biggr) \Biggr),n \geq1 \Biggr\rbrace,
\]
we will rely
on a technical result from \cite{B99}, which gives a bound on the
supremum of a general random walk in terms of bounds on its increments.
In particular, the following is shown in \cite{B99}, Theorem 10.2:
%
%le3 #&#
\begin{lemma}\label{billingsleylemma}
Suppose $k < \infty$, $X_1,X_2,\ldots,X_k$ is a sequence of general
(possibly dependent and not identically distributed) random variables,
$S_j \stackrel{\Delta}{=}\break \sum_{i=1}^j X_i$ and
$M_k = \max_{j \leq k} |S_j|$. Further suppose that there exist real
numbers $\alpha> \frac{1}{2}$, \mbox{$\beta\geq0$}, and a sequence of
nonnegative numbers $u_1,u_2,\ldots,u_k$ s.t.
for all $0 \leq i \leq j \leq k$ and $x > 0$,
\[
\pr\bigl( | S_j - S_i | \geq x \bigr) \leq x^{- 4 \beta}
\biggl( \sum_{i < l \leq j} u_l
\biggr)^{2 \alpha}.
\]
Then there exists a finite constant $K_{\alpha,\beta}$, depending
only on $\alpha$ and $\beta$, s.t. for all $x > 0$,
\[
\pr( M_k \geq x ) \leq K_{\alpha,\beta} x^{- 4 \beta} \biggl(
\sum_{0 < l \leq k} u_l \biggr)^{2 \alpha}.
\]
\end{lemma}

We will also use frequently the inequality
%
%e6 #&#
\begin{equation}
\label{convexlemma} (x_1 + x_2)^r
\leq2^{r-1} x_1^r + 2^{r-1}
x_2^r \qquad\mbox{for all } r \geq1 \mbox{ and }
x_1,x_2 \geq0,
\end{equation}
which follows from the convexity of $f(x) \stackrel{\Delta}{=} x^r$,
$r \geq1$.

Before proceeding with the proof of Theorem \ref{tightc}, we
establish two more auxiliary results. The first bounds the moments of
the sum of $n$ i.i.d. zero-mean r.v. in terms of the moments of the
individual r.v.s and $n$, and is proven in \cite{W60}.
%
%le4 #&#
\begin{lemma}\label{csumbound}
For all $r \geq2$, there exists $C_{r} < \infty$ (depending only on
$r$) s.t. for all r.v. $X$ satisfying $\E[X] = 0$ and $\E[|X|^{r}]
< \infty$,
if $\lbrace X_i, i \geq1 \rbrace$ is a sequence of i.i.d. r.v.s
distributed as $X$, then for all $k \geq1$,
\[
\E\Biggl[\Biggl|\sum_{i=1}^k
X_i\Biggr|^{r}\Biggr] \leq C_{r} k^{ {r}/{2}} \E
\bigl[|X|^{r}\bigr].
\]
\end{lemma}
Second, we prove a bound for the central moments of a pooled
equilibrium renewal process.
%
%le5 #&#
\begin{lemma}\label{binomial2}
Let $X$ denote any nonnegative r.v. s.t. $\E[X] = \mu^{-1} \in
(0,\infty)$, and $\E[X^r] < \infty$ for some $r \geq2$. Let
$\lbrace Z^e_i (t), i \geq1 \rbrace$ denote a set of i.i.d.
equilibrium renewal processes with renewal distribution $X$. Then there
exists $C_{X,r} < \infty$ (depending only on $X$ and $r$) s.t. for all
$n \geq1$ and $t \geq0$,
%
%e7 #&#
\begin{equation}
\label{focus1} \E\Biggl[ \Biggl| \sum_{i=1}^n
Z^e_i(t) - \mu n t \Biggr|^r \Biggr] \leq
C_{X,r} \bigl( 1 + (n t)^{{r/2}} \bigr).
\end{equation}
\end{lemma}
\begin{pf}
The proof is deferred to the \hyperref[app]{Appendix}.
\end{pf}\eject

With the above bounds at our disposal, we now complete the proof of
Theorem~\ref{tightc}.
\begin{pf*}{Proof of Theorem \ref{tightc}}
In the notation of Theorem \ref{statwait2}, let
\begin{eqnarray*}
\zeta_{n,k} &\stackrel{\Delta} {=}& 1 - \sum
_{i=1}^n N_i\Biggl( \lambda
_n^{-1} \sum_{j=1}^{k-1}
A^1_j, \lambda_n^{-1} \sum
_{j=1}^{k} A^1_j \Biggr),
\\
W_{n,k} &\stackrel{\Delta} {=}& k - \sum_{i=1}^n
N_i\Biggl( \lambda_n^{-1} \sum
_{j=1}^{k} A^1_j \Biggr).
\end{eqnarray*}
That $\lbrace\zeta_{n,i}, i \geq1 \rbrace$ is a stationary,
countably infinite sequence of r.v. follows from the stationary
increments property of the equilibrium renewal process.
Since $\E[\sum_{i=1}^n N_i(t)] = n t \mu$ for all $t \geq0$, it
follows that $a_n \stackrel{\Delta}{=} \E[ \zeta_{n,1}] = 1 - \frac
{n}{\lambda_n} = - \frac{B}{n^{{1/2}} - B } < 0$, and $\lim_{n
\rightarrow\infty} a_n = 0$.
Thus we need only verify assumptions (i) and (ii) of Theorem \ref
{statwait2}. Since $\E[A^{2 + \varepsilon}],\E[S^{2 + \varepsilon}] <
\infty$ for some $\varepsilon> 0$ by the $T_0$ assumptions, we may fix
some $r > 2$ s.t. $\E[A^r],\E[S^r] < \infty$. Note that
%
%e8 #&#
%e9 #&#
\begin{eqnarray}
&&
\E\bigl[ | W_{n,k} - k a_n |^{r} \bigr] \nonumber\\
&&\qquad= \E
\Biggl[\Biggl| \sum_{i=1}^n N_i
\Biggl(\lambda_n^{-1} \sum_{j=1}^{k}
A^1_j\Biggr) - \frac{k n}{\lambda_n} \Biggr|^r \Biggr]
\nonumber
\\
&&\qquad\leq \E\Biggl[ \Biggl( \Biggl| \sum_{i=1}^n
N_i\Biggl(\lambda_n^{-1} \sum
_{j=1}^{k} A^1_j\Biggr) - \mu
\frac{n}{\lambda_n} \sum_{j=1}^{k}
A^1_j \Biggr| + \Biggl|\mu\frac{n}{\lambda_n} \sum
_{j=1}^{k} A^1_j -
\frac{k
n}{\lambda_n} \Biggr| \Biggr)^r \Biggr]
\nonumber
\\
\label{ainside1}
&&\qquad\leq 2^{r-1} \E\Biggl[\Biggl|\sum_{i=1}^n
N_i\Biggl(\lambda_n^{-1} \sum
_{j=1}^{k} A^1_j\Biggr) - \mu
\frac{n}{\lambda_n} \sum_{j=1}^{k}
A^1_j\Biggr|^r \Biggr]
\\
\label{ainside2}
&&\qquad\quad{} + 2^{r-1} \E\Biggl[\Biggl|\mu\frac{n}{\lambda_n} \sum
_{j=1}^{k} A^1_j -
\frac{k n}{\lambda_n}\Biggr|^r \Biggr] \qquad\mbox{by (\ref{convexlemma})}.
\end{eqnarray}
We now bound (\ref{ainside1}). By Lemmas \ref{csumbound}--\ref
{binomial2}, there exist $C_{S,r}, C_r < \infty$ independent of $n$,
and $k$ s.t. $\E[|\sum_{i=1}^n N_i(\lambda_n^{-1} \sum_{j=1}^{k}
A^1_j) - \mu\frac{n}{\lambda_n} \sum_{j=1}^{k}
A^1_j|^r ]$ is at most
%
%e10 #&#
\begin{eqnarray}\label{aabbc1}
&& C_{S,r} + C_{S,r} \biggl(\frac{n}{\lambda_n}
\biggr)^{{r/2}} \E\Biggl[ \Biggl( \sum_{j=1}^k
A^1_j \Biggr)^{{r/2}} \Biggr] \qquad\mbox{by Lemma
\ref{binomial2}}
\nonumber
\\
&&\qquad\leq C_{S,r} + C_{S,r} \biggl(\frac{n}{\lambda_n}
\biggr)^{{r/2}} \E\Biggl[ \Biggl( \Biggl|\sum_{j=1}^k
\bigl( A^1_j - \mu^{-1} \bigr)\Biggr| + k \mu
^{-1} \Biggr)^{{r/2}} \Biggr]
\nonumber
\\
&&\qquad\leq C_{S,r} + C_{S,r} \biggl(\frac{n}{\lambda_n}
\biggr)^{{r/2}} \Biggl(2^{{r/2}-1} \E\Biggl[\Biggl|\sum
_{j=1}^k \bigl( A^1_j - \mu
^{-1} \bigr)\Biggr|^{ {r}/{2} } \Biggr] \nonumber\\
&&\qquad\quad\hspace*{148pt}{}+ 2^{{r/2}-1} \bigl(k \mu
^{-1}\bigr)^{{r/2}} \Biggr)
\nonumber
\\
&&\qquad\leq C_{S,r} + 2^{{r/2}-1} C_{S,r} \biggl(
\frac{n}{\lambda
_n}\biggr)^{{r/2}} \Biggl(\E^{{1/2}} \Biggl[\Biggl|\sum
_{j=1}^k \bigl(A^1_j
- \mu^{-1}\bigr)\Biggr|^{r} \Biggr] + \bigl(k \mu^{-1}
\bigr)^{{r/2}} \Biggr)
\\
&&\qquad\quad \mbox{since } \E[X] \leq\E^{{1/2}}\bigl[X^2\bigr] \mbox{ for
any nonnegative r.v. } X
\nonumber
\\
&&\qquad\leq C_{S,r} + 2^{ {r}/{2} - 1} C_{S,r} \biggl(
\frac{n}{\lambda
_n}\biggr)^{{r/2}} \bigl( \bigl( C_r
k^{{r/2}} \E\bigl[\bigl|A - \mu^{-1}\bigr|^r \bigr]
\bigr)^{{1/2}} + \bigl(k \mu^{-1}\bigr)^{{r/2}} \bigr)
\nonumber
\\
&&\qquad\quad \mbox{by Lemma \ref{csumbound}}
\nonumber
\\
&&\qquad\leq C'_1 k^{{r/2}}\nonumber
\end{eqnarray}
for some finite constant $C'_1$ independent of $n$ and $k$, since $ \E
[|A - \mu^{-1}|^r ] < \infty$, and $\lim_{n \rightarrow
\infty} \frac{n}{\lambda_n} = 1$.

We now bound (\ref{ainside2}).
%
%e11 #&#
\begin{eqnarray}\label{aabbc2}
&&
\E\Biggl[\Biggl|\mu\frac{n}{\lambda_n} \sum_{j=1}^{k}
A^1_j - \frac{k
n}{\lambda_n}\Biggr|^r \Biggr] \nonumber\\
&&\qquad=
\biggl(\frac{n}{\lambda_n}\biggr)^r \mu^{r} \E\Biggl[\Biggl|\sum
_{j=1}^k \bigl( A^1_j
- \mu^{-1} \bigr) \Biggr|^r \Biggr]
\nonumber\\[-8pt]\\[-8pt]
&&\qquad\leq \biggl( C_r \biggl(\frac{n}{\lambda_n}\biggr)^r
\mu^{r} \E\bigl[\bigl|A - \mu^{-1}\bigr|^r \bigr] \biggr)
k^{{r/2}} \qquad\mbox{by Lemma \ref{csumbound}}
\nonumber\\
&&\qquad\leq C''_1 k^{{r/2}} \qquad\mbox{for
some finite constant } C''_1
\mbox{ independent of } n \mbox{ and } k.\nonumber
\end{eqnarray}
Using (\ref{aabbc1}) to bound (\ref{ainside1}) and (\ref{aabbc2}) to
bound (\ref{ainside2}), it follows that assumption (i) of Theorem
\ref{statwait2} holds for the finite constant $C_1 \stackrel{\Delta
}{=} 2^{r-1} (C'_1 + C''_1)$. We now apply Lemma \ref
{billingsleylemma} to show that assumption (ii) holds as well.
In the notation of Lemma \ref{billingsleylemma}, let $S_{n,i}
\stackrel{\Delta}{=} W_{n,i} - i a_n$ for $i \geq0$, and $M_{n,k}
\stackrel{\Delta}{=}
\max_{i \leq k} | W_{n,i} - i a_n |$ for $k \geq0$. Then for all $n$,
$0 \leq i \leq j$ and $x > 0$,
\begin{eqnarray*}
\pr\bigl( | S_{n,j} - S_{n,i} | \geq x \bigr) &=& \pr\bigl(
|S_{n,j-i}| \geq x \bigr) \qquad\mbox{by stationary increments}
\\
&=& \pr\bigl( \bigl|W_{n,j-i} - (j- i) a_n\bigr| \geq x \bigr)
\\
&\leq& C_1 (j-i)^{ {r}/{2} } x^{-r} \qquad\mbox{by
Markov's inequality}
\\
&\leq& \bigl( (C_1+1) (j-i) \bigr)^{ {r}/{2} } x^{-r}.
\end{eqnarray*}
Thus for all $n$ and $k \geq1$, we may apply Lemma \ref
{billingsleylemma} (in the notation of Lemma~\ref{billingsleylemma})
with $\beta\stackrel{\Delta}{=} \frac{r}{4}, \alpha\stackrel
{\Delta}{=} \frac{r}{4}$, and $u_l \stackrel{\Delta}{=} (C_1 +
1)$ for $1 \leq l \leq k$, to find that
there exists a constant $K_r < \infty$ (depending only on $r$) s.t.
for all $x > 0$,
%
%e12 #&#
\begin{equation}
\label{ass4holds} \pr\Bigl( \max_{i=1,\ldots,k} (W_{n,i} - i
a_n) > x \Bigr) \leq K_r (C_1+1)^{ {r}/{2} }
k^{{r/2}} x^{-r}.
\end{equation}
It follows that assumption (ii) of Theorem \ref{statwait2} holds as
well, with (in the notation of Theorem \ref{statwait2}) $C_2
\stackrel{\Delta}{=} K_r (C_1+1)^{ {r}/{2} }$, $\varepsilon
\stackrel{\Delta}{=} r - 2$. Combining the above, we find that all
assumptions of Theorem \ref{statwait2} hold, and thus we may apply
Theorem \ref{statwait2} to find that
\[
\Biggl\lbrace\frac{B}{n^{{1/2}} - B} \sup_{k \geq0} \Biggl( k -
\sum
_{i=1}^n N_i\Biggl(
\lambda_n^{-1} \sum_{j=1}^k
A^1_j \Biggr) \Biggr), n \geq1 \Biggr\rbrace
\]
is tight. Combining with (\ref{tightenit}) completes the proof of
Theorem \ref{tightc}.
\end{pf*}
%
%s6 #&#
\section{\texorpdfstring{Large deviation results and proof of Theorem \protect\ref{ld1}}{Large deviation results and proof of Theorem 2}}\label{ldsecc}
In this section, we complete the proofs of our main results. We proceed
by combining our upper and lower bounds with several known weak
convergence results for (pooled) renewal processes and the suprema of
Gaussian processes. Recall that a Gaussian process on $\reals$ is a
stochastic process $Z(t)_{t \geq0}$ s.t.
for any finite set of times $t_1,\ldots,t_k$, the vector $ (
Z(t_1),\ldots,Z(t_k) )$ has a Gaussian distribution. A Gaussian
process $Z(t)$ is known to be uniquely determined by its mean function
$\E[Z(t)]$ and covariance function $\E[Z(s) Z(t)]$, and refer the
reader to \cite{D44,IR78,A90,marcus2006markov}, and the references therein for details on
existence, continuity, etc.

%s6.1 #&#
\subsection{Preliminary weak convergence results}\label{weaksubsec}
In this subsection we review several weak convergence results for
renewal processes, and apply them to $A_n(t)$ and $\sum_{i=1}^n
N_i(t)$. For an excellent review of weak convergence, and the
associated spaces (e.g., $D[0,T]$) and topologies/metrics (e.g.,
uniform, $J_1$), the reader is referred to \cite{W02}. Let ${\mathcal
A}(t)$ denote the w.p.1 continuous Gaussian process s.t. $\E[{\mathcal
A}(t)] = 0, \E[{\mathcal A}(s){\mathcal A}(t)] = \mu c^2_A \min
(s,t)$, namely ${\mathcal A}(t)$ is a driftless Brownian motion. Then
from the well-known functional central limit theorem (FCLT) for renewal
processes (see \cite{B99}, Theorem 14.6), we have the following:
%
%th6 #&#
\begin{theorem}\label{renewalfclt}
For any $T \in[0,\infty)$, the sequence of processes $\lbrace\lambda
_n^{-{1/2}} ( A_n(t) - \lambda_n \mu t )_{0 \leq t
\leq T}, n \geq1 \rbrace$ converges weakly to
${\mathcal A}(t)_{0 \leq t \leq T}$ in the space $D[0,T]$ under the
$J_1$ topology.
\end{theorem}
We now give a weak convergence result for $\sum_{i=1}^n N_i(t)$, which
is stated in \cite{W02} (see Theorem 7.2.3) and formally proven in
\cite{Whitt85} (see Theorem 2).
%
%th7 #&#
\begin{theorem}\label{renewalfclt2}
There exists a w.p.1 continuous Gaussian process ${\mathcal D}(t)$ s.t.
$\E[{\mathcal D}(t)] = 0, \E[{\mathcal D}(s){\mathcal D}(t)] = \E
[ ( N_1(s) - \mu s ) ( N_1(t) - \mu t )]$ for all
$s,t \geq0$. Furthermore, for any $T \in[0,\infty)$, the sequence of
processes $\lbrace n^{-{1/2}} ( \sum_{i=1}^n N_i(t) - n
\mu t )_{0 \leq t \leq T}, n \geq1 \rbrace$ converges weakly to
${\mathcal D}(t)_{0 \leq t \leq T}$ in the space $D[0,T]$ under the
$J_1$ topology.
\end{theorem}
We note that the $T_0$ assumptions \ref{t1} and \ref{t3}, which
guarantee that $\E[S^{2 + \varepsilon}] < \infty$ and $\limsup_{x
\downarrow0} x^{-1} \pr(S \leq x) < \infty$, ensure that the
technical conditions required to apply \cite{W02}, Theorem 7.2.3,
namely that $E[S^2] < \infty$ and $\limsup_{x \downarrow0} x^{-1}
( \pr(S \leq x) - \pr(S = 0) ) < \infty$, hold.

Then from Theorems \ref{renewalfclt}--\ref
{renewalfclt2}, we have the following:
%
%le6 #&#
\begin{lemma}\label{renewalfclt3}
For any fixed $T \geq0$, $\lbrace n^{-{1/2}} ( A_n(t) -
\sum_{i=1}^n N_i(t) )_{0 \leq t \leq T}, n \geq1 \rbrace$
converges weakly to
$ ( {\mathcal A}(t) - {\mathcal D}(t) - B \mu t )_{0 \leq t
\leq T}$ in the space $D[0,T]$ under the $J_1$ topology.
\end{lemma}
\begin{pf}
Note that
\begin{eqnarray*}
&&
n^{-{1/2}} \Biggl( A_n(t) - \sum_{i=1}^n
N_i (t) \Biggr)_{0 \leq
t \leq T} \\
&&\qquad= \Biggl( \lambda_n^{{1/2}}
n^{-{1/2}} \bigl( A_n(t) - \lambda_n \mu t \bigr)
\lambda_n^{-{1/2}} \\
&&\qquad\quad\hspace*{4.5pt}{}- \Biggl( \sum_{i=1}^n
N_i(t) - n \mu t \Biggr)n^{-{1/2}}- B \mu t
\Biggr)_{0
\leq t \leq T}.
\end{eqnarray*}
The lemma then follows from Theorems \ref{renewalfclt} and \ref{renewalfclt2}.
\end{pf}
We note that a process very similar to $ ( {\mathcal A}(t) -
{\mathcal D}(t) - B \mu t )_{0 \leq t \leq T}$ was studied in
\cite{Whitt85} as the weak limit of a sequence of queues with
superposition arrival processes. The continuity of the supremum map in
the space $D[0,T]$ under the $J_1$ topology (see \cite{W02}, Theorem
13.4.1), combined with Lemma \ref{renewalfclt3}, implies the
following:
%
%co2 #&#
\begin{corollary}\label{renewalfclt4}
For any fixed\vspace*{1pt} $T \geq0$,
$\lbrace n^{-{1/2}} \sup_{ 0 \leq t \leq T} (A_n(t) - \sum
_{i=1}^n N_i(t) )$, $n \geq1 \rbrace$ converges weakly to
the r.v. $\sup_{ 0 \leq t \leq T} ( {\mathcal A}(t) - {\mathcal
D}(t) - B \mu t )$.
\end{corollary}
%
%s6.2 #&#
\subsection{Preliminary large deviation results}\label{ldsec}
Before proceeding with the remaining proofs, we will need to establish
some results from the theory of large deviations of Gaussian processes
and their suprema. We note that the relationship between the large
deviations of suprema of Gaussian processes and the large deviations of
queueing systems is well known (see \cite{DO95,D05}), and
there is a significant literature studying the large deviations of such
processes (e.g., \cite{DO95,DLS03,DMR98,Deb99,D05}).
We will rely heavily on the following theorem, proven in \cite{D05}
(in a more general form).
%
%th8 #&#
\begin{theorem}\label{largedev1}
Suppose ${\mathcal Z}(t)$ is a centered, separable Gaussian process
with stationary increments, s.t. $E[{\mathcal Z}^2(t)]$ is a continuous
function of $t$ on $[0,\infty)$,
$\lim_{t \downarrow0} ( E[{\mathcal Z}^2(t)] \log^2(t) ) =
0$ and
$\lim_{t \rightarrow\infty} t^{-1} \E[ {\mathcal Z}^2(t) ] = \sigma
^2 > 0$. Then for any $c > 0$,
\[
\lim_{x \rightarrow\infty} x^{-1} \log\Bigl( \pr\Bigl( \sup
_{t
\geq0} \bigl( {\mathcal Z}(t) - c t \bigr) \geq x \Bigr) \Bigr)
= - \frac{ 2 c
}{\sigma^2}.
\]
\end{theorem}
Also implicit from the discussion in \cite{DO95} (although we include
a short proof) is the following:
%
%th9 #&#
\begin{theorem}\label{largedev2}
Under the same assumptions as Theorem \ref{largedev1}, for any $c > 0$,
\[
\lim_{x \rightarrow\infty} x^{-1} \log\Bigl( \sup
_{t \geq0} \pr\bigl( {\mathcal Z}(t) - c t > x \bigr) \Bigr) = -
\frac{ 2 c
}{\sigma^2}.
\]
\end{theorem}
\begin{pf}
That $\limsup_{x \rightarrow\infty} x^{-1} \log( \sup_{t
\geq0} \pr( {\mathcal Z}(t) - c t > x ) ) \leq-
\frac{ 2 c }{\sigma^2}$ follows immediately from Theorem \ref
{largedev1} and the fact that
$\sup_{t \geq0} \pr( {\mathcal Z}(t) - c t > x ) \leq\pr
( \sup_{t \geq0} ( {\mathcal Z}(t) - c t ) > x )$.\vspace*{2pt}

Letting $t = \frac{x}{c}$, we find that
%
%e13 #&#
\begin{equation}
\label{aab1} \sup_{t \geq0} \pr\bigl( {\mathcal Z}(t) - c t > x
\bigr) \geq\pr\biggl( {\mathcal Z}\biggl( \frac{x}{c} \biggr) - x > x
\biggr).
\end{equation}
Let $G$ denote a normally distributed r.v. with mean 0 and variance 1.
Then since ${\mathcal Z}(\frac{x}{c})$ is normally distributed with
mean zero, it follows from (\ref{aab1}) that
%
%e14 #&#
\begin{equation}
\label{aab2} \sup_{t \geq0} \pr\bigl( {\mathcal Z}(t) - c t > x
\bigr) \geq\pr\biggl( G > 2 x \E^{-{1/2}} \biggl[ {\mathcal
Z}^2 \biggl(\frac{x}{c}\biggr) \biggr] \biggr).
\end{equation}
We use the following identity from \cite{AS72}, equation 7.1.13.
Namely, for all $y > 0$,
\[
\pr( G > y ) \geq\bigl( y + \bigl(y^2 + 4\bigr)^{-{1/2}}
\bigr)^{-1} \biggl(\frac{2}{\pi} \biggr)^{{1/2}} \exp\biggl(
- \frac{y^2}{2} \biggr).
\]
Thus
%
%e15 #&#
\begin{equation}
\label{normalbound1} \pr( G > y ) \geq\exp\biggl( - \frac{y^2}{2} -
y\biggr)
\qquad\mbox{for all sufficiently large } y.
\end{equation}
By assumption, $\lim_{t \rightarrow\infty} t^{-1} \E[ {\mathcal
Z}^2(t) ] = \sigma^2 > 0$, and thus
\[
\lim_{x \rightarrow\infty} 2 x \E^{-{1/2}} \biggl[ {\mathcal
Z}^2 \biggl(\frac{x}{c}\biggr) \biggr]
= \infty.
\]
It thus follows from (\ref{aab2}) and (\ref
{normalbound1}) that for all sufficiently large $x$,
\[
x^{-1} \log\Bigl( \sup_{t \geq0} \pr\bigl( {\mathcal
Z}(t) - c t > x \bigr) \Bigr) \geq- 2 x \E^{-1} \biggl[ {\mathcal
Z}^2 \biggl(\frac{x}{c}\biggr) \biggr] - 2 \E^{-{1/2}}
\biggl[ {\mathcal Z}^2 \biggl(\frac{x}{c}\biggr) \biggr].
\]
Since $\lim_{x \rightarrow\infty} (\frac{x}{c})^{-1} \E[ {\mathcal
Z}^2 (\frac{x}{c})] = \sigma^2$, it follows that
\[
\liminf_{x \rightarrow\infty} x^{-1} \log\Bigl( \sup_{t \geq0}
\pr\bigl( {\mathcal Z}(t) - c t > x \bigr) \Bigr) \geq- \frac
{2c}{\sigma^2},
\]
completing the proof of the theorem.
\end{pf}

In light of Theorem \ref{largedev1}, Theorem \ref{largedev2} can be
interpreted as saying that such a process is ``most likely'' to exceed
a given value $x$ at a particular\vspace*{1pt} time (roughly $\frac{x}{c}$), and
much less likely to exceed that value at any other time; see the
discussion in \cite{DO95}. We note that the duality of Theorems \ref
{largedev1} and \ref{largedev2} coincides with the duality exhibited by
our upper and lower bounds (Theorems \ref{ubound1} and~\ref{lbound1})---a
relationship that we will exploit to prove our large
deviation results.

We are now in a position to apply Theorems \ref
{largedev1} and \ref{largedev2} to ${\mathcal A}(t) - {\mathcal D}(t)$.
%
%co3 #&#
\begin{corollary}\label{largedev3}
\textup{(i)} $\lim_{x \rightarrow\infty} x^{-1} \log\pr( \sup_{t
\geq0} ( {\mathcal A}(t) - {\mathcal D}(t) - B \mu t ) \geq
x ) = - 2 B ( c^2_A + c^2_S )^{-1}$;

\textup{(ii)} $\lim_{x \rightarrow\infty} x^{-1} \log( \sup_{t \geq
0} \pr( {\mathcal A}(t) - {\mathcal D}(t) - B \mu t > x )
) = - 2 B ( c^2_A + c^2_S )^{-1}$.
\end{corollary}
\begin{pf}
That ${\mathcal A}(t) - {\mathcal D}(t)$ is a centered, separable
Gaussian process with stationary increments follows from definitions,
the independence of ${\mathcal A}(t)$ and ${\mathcal D}(t)$ and the
fact that both ${\mathcal A}(t)$ and $N_1(t)$ have stationary increments.
The independence of ${\mathcal A}(t)$ and ${\mathcal D}(t)$ implies that
%
%e16 #&#
\begin{equation}
\label{eqzz1} \E\bigl[ \bigl( {\mathcal A}(t) - {\mathcal D}(t)
\bigr)^2\bigr] = \mu c^2_A t + \E\bigl[
\bigl( N_1(t) - \mu t \bigr)^2\bigr].
\end{equation}
We now prove that
%
%e17 #&#
\begin{equation}
\label{eqzz2} \lim_{t \rightarrow\infty} t^{-1} \E\bigl[ \bigl( {
\mathcal A}(t) - {\mathcal D}(t) \bigr)^2\bigr] = \mu
\bigl(c^2_A + c^2_S\bigr).
\end{equation}
Let $G_S$ denote a normally distributed r.v. with mean 0 and variance
$\mu c^2_S$. It follows from the well-known\vspace*{-1pt} central limit theorem for
renewal processes (see \cite{Ross96}, Theorem 3.3.5), and the fact
that $h(z) \stackrel{\Delta}{=} z^2$ is a continuous function, that
the sequence of r.v.s $\lbrace( t^{-{1/2}} ( N_1(t) -
\mu t ) )^2, t \geq1 \rbrace$ converges weakly to
$G^2_S$. Recall that $\E[S^{2 + \varepsilon}] < \infty$ for some
$\varepsilon> 0$ by the $T_0$ assumptions.
Thus it follows from Lemma~\ref{binomial2} that the sequence of r.v.s
$\lbrace( t^{-{1/2}} ( N_1(t) - \mu t )
)^2, t \geq1 \rbrace$ is uniformly integrable. It follows that $\lim_{t
\rightarrow\infty} t^{-1} \E[ ( N_1(t) - \mu t )^2] =
\mu c^2_S$, since
uniform integrability plus weak convergence implies convergence of
moments, completing the proof of (\ref{eqzz2}).

We now prove that $\lim_{t \downarrow0} ( \E[ (
{\mathcal A}(t) - {\mathcal D}(t) )^2] \log^2(t) ) = 0$.
Note that for all $t \geq0$ and $k \geq1$, ${\mathcal P}(N_1(t) \geq
k) \leq{\mathcal P} ( R(S) \leq t ) ( {\mathcal P}(S
\leq t) )^{k-1}$. It thus follows from the $T_0$ assumptions (and
a straightforward analogy to an appropriate geometrically distributed
r.v.) that $\limsup_{t \downarrow0} t^{-1} \E[ ( N_1(t) - \mu
t )^2] < \infty$. Combining the above completes the proof. In
addition, the continuity of $\E[ ( {\mathcal A}(t) - {\mathcal
D}(t) )^2]$ on $[0,\infty)$ follows from the above and a simple
application of the Cauchy--Schwarz inequality. It follows that
${\mathcal A}(t) - {\mathcal D}(t)$ satisfies the conditions needed to
apply Theorems \ref{largedev1} and \ref{largedev2}, from which the
corollary follows.
\end{pf}
%
%s6.3 #&#
\subsection{\texorpdfstring{Proof of Theorem \protect\ref{ld1}}{Proof of Theorem 2}}
Before completing the proofs of our main results, it will be useful to
prove a strengthening of Theorem \ref{tightc}. Namely, we have the following:
%
%le7 #&#
\begin{lemma}\label{largeT1}
For all $x \geq0$,
%
%e18 #&#
\begin{equation}
\label{fixme1a} \limsup_{T \rightarrow\infty} \limsup_{n \rightarrow
\infty}
\pr\Biggl( n^{-{1/2}} \sup_{ t \geq T} \Biggl(A_n(t)
- \sum_{i=1}^n N_i(t) \Biggr)
> x \Biggr) = 0.
\end{equation}
\end{lemma}
\begin{pf}
Note that since $x \geq0$, it follows from a simple union bound and
stationary increments that the left-hand side of (\ref{fixme1a}) is at most
%
%e19 #&#
%e20 #&#
\begin{eqnarray}
\label{Tlimeq1}
&& \limsup_{T \rightarrow\infty} \limsup_{n \rightarrow\infty} \pr
\Biggl(
n^{-{1/2}} \Biggl( A_n(T) - \sum_{i=1}^n
N_i(T) \Biggr) > - \frac{B}{2} \mu T \Biggr)
\\
\label{Tlimeq2}
&&\qquad{}+ \limsup_{T \rightarrow\infty} \limsup_{n \rightarrow\infty} \pr
\Biggl( n^{-{1/2}} \sup_{ t \geq0} \Biggl(A_n(t) -
\sum_{i=1}^n N_i(t) \Biggr) >
\frac{B}{2} \mu T \Biggr).
\end{eqnarray}
It follows from Lemma \ref{renewalfclt3}, (\ref{eqzz2}) and the
central limit theorem that (\ref{Tlimeq1})\break equals zero.
As our proof of Theorem \ref{tightc} demonstrates tightness of\break $
\lbrace n^{-{1/2}} \sup_{ t \geq0} (A_n(t) - \sum_{i=1}^n N_i(t)
), n \geq1 \rbrace$, it follows that (\ref
{Tlimeq2}) also equals zero.
Combining the above completes the proof.
\end{pf}
We now complete the proof of Theorem \ref{ld1}.
\begin{pf*}{Proof of Theorem \ref{ld1}}
We first prove the upper bound. By Lemma~\ref{largeT1}, for any $x > 0$,
we may construct a strictly increasing sequence of integers $\lbrace
T_{x,k^{-1}}, k \geq1 \rbrace$ s.t. for all $k \geq1$,
\[
\limsup_{n \rightarrow\infty} \pr\Biggl( n^{-{1/2}} \sup
_{
t \geq T_{x,k^{-1}} } \Biggl( A_n(t) - \sum
_{i=1}^n N_i(t) \Biggr) \geq x \Biggr) <
k^{-1}.
\]
It follows that for all $x> 0$ and $k \geq1$,
%
%e21 #&#
\begin{eqnarray}
\label{eqeq2}
&&
\limsup_{n \rightarrow\infty} \pr\Biggl( n^{-
{1/2}} \sup
_{
t \geq0 } \Biggl(A_n(t) - \sum
_{i=1}^n N_i(t) \Biggr) \geq x \Biggr)
\nonumber\\[-8pt]\\[-8pt]
&&\qquad\leq\limsup_{n \rightarrow\infty} \pr\Biggl( n^{-{1/2}} \sup
_{ 0 \leq t \leq T_{x,k^{-1}} } \Biggl(A_n(t) - \sum
_{i=1}^n N_i(t) \Biggr) \geq x \Biggr) +
k^{-1}.\nonumber
\end{eqnarray}
By the Portmanteau theorem (see \cite{B99}), a sequence of r.v.s
$\lbrace X_n \rbrace$ converges weakly to the r.v. $X_{\infty}$
if and only if for all closed subsets $C$ of $\reals$,\break $\limsup_{n
\rightarrow\infty} \pr(X_n \in C) \leq\pr(X_{\infty} \in C)$ if
and only if for all open subsets $O$ of $\reals$, $\pr(X_{\infty}
\in O) \leq\liminf_{n \rightarrow\infty} \pr(X_n \in O)$. It
follows from (\ref{eqeq2}) and Corollary \ref{renewalfclt4} that for
all $x > 0$ and $k \geq1$,
%
%e22 #&#
\begin{eqnarray}
\label{eqeq3}
&&
\limsup_{n \rightarrow\infty} \pr\Biggl( n^{-
{1/2}} \sup
_{
t \geq0 } \Biggl(A_n(t) - \sum
_{i=1}^n N_i(t) \Biggr) \geq x \Biggr)\nonumber\\[-8pt]\\[-8pt]
&&\qquad
\leq\pr\Bigl( \sup_{ 0 \leq t \leq T_{x,k^{-1}} } \bigl( {\mathcal
A}(t) - {\mathcal
D}(t) - B \mu t \bigr) \geq x \Bigr) + k^{-1}.\nonumber
\end{eqnarray}
Note that the sequence of events $ \lbrace\sup_{ 0
\leq t \leq T_{x,k^{-1}} } ( {\mathcal A}(t) - {\mathcal D}(t) - B
\mu t ) \geq x,\break k \geq1 \rbrace$ is monotone in $k$. It
follows that
\begin{eqnarray*}
&&
\lim_{k \rightarrow\infty} \pr\Bigl( \sup_{ 0 \leq
t \leq T_{x,k^{-1}} } \bigl( {
\mathcal A}(t) - {\mathcal D}(t) - B \mu t \bigr) \geq x \Bigr) \\
&&\qquad= \pr
\Bigl( \sup
_{ t \geq0 } \bigl( {\mathcal A}(t) - {\mathcal D}(t) - B \mu t
\bigr) \geq x \Bigr).
\end{eqnarray*}
It then follows from (\ref{eqeq3}), by letting $k \rightarrow\infty
$, that for all $x > 0$,
%
%e23 #&#
\begin{eqnarray}
\label{eqeq4}
&&
\limsup_{n \rightarrow\infty} \pr\Biggl( n^{-
{1/2}} \sup
_{
t \geq0 } \Biggl(A_n(t) - \sum
_{i=1}^n N_i(t) \Biggr) \geq x \Biggr)\nonumber\\[-8pt]\\[-8pt]
&&\qquad
\leq\pr\Bigl( \sup_{ t \geq0 } \bigl( {\mathcal A}(t) - {\mathcal
D}(t) - B \mu t \bigr) \geq x \Bigr).\nonumber
\end{eqnarray}
From Theorem \ref{ubound1} and (\ref{eqeq4}) we have
\begin{eqnarray*}
&& \limsup_{x \rightarrow\infty} x^{-1} \log\Bigl( \limsup
_{n
\rightarrow\infty} \pr\bigl( \bigl( Q^n(\infty) - n
\bigr)^+n^{-{1/2}} \geq x \bigr) \Bigr)
\\
&&\qquad\leq \limsup_{x \rightarrow\infty} x^{-1} \log\pr\Bigl( \sup
_{ t \geq
0 } \bigl( {\mathcal A}(t) - {\mathcal D}(t) - B \mu t
\bigr) \geq x \Bigr)
\\
&&\qquad= - 2 B \bigl(c^2_A + c^2_S
\bigr)^{-1} \qquad\mbox{by Corollary \ref{largedev3}(i),}
\end{eqnarray*}
which completes the proof of the upper bound.

We now complete the proof of Theorem \ref{ld1} by
demonstrating that if
$A$ is an exponentially distributed r.v., then
%
%e24 #&#
\begin{equation}
\label{lldd1} \liminf_{x \rightarrow\infty} x^{-1} \log\Bigl(
\liminf_{n
\rightarrow\infty} \pr\bigl( \bigl( Q^n(\infty) - n
\bigr)^+n^{-{1/2}} > x \bigr) \Bigr) \geq-2 B \bigl(c^2_A
+ c^2_S\bigr)^{-1}.\hspace*{-28pt}
\end{equation}
Let $Z_n$ denote a Poisson r.v. with mean $\lambda_n$. It follows from
Theorem \ref{lbound1} that for all $x > 0$,
%
%e25 #&#
\begin{eqnarray}
\label{lboundpleaseuse}
&&\liminf_{n \rightarrow\infty} \pr\bigl( \bigl(
\cal{Q}^n(\infty) - n \bigr)^+n^{-{1/2}} > x \bigr)
\nonumber\\
&&\qquad\geq\Bigl( \liminf_{n \rightarrow\infty} \pr( Z_n \geq n )
\Bigr) \\
&&\qquad\quad{}\times\Biggl( \liminf_{n \rightarrow\infty} \sup_{t \geq0} \pr
\Biggl( n^{-{1/2}} \Biggl(A_n(t) - \sum
_{i=1}^n N_i(t) \Biggr) > x \Biggr)
\Biggr).\nonumber
\end{eqnarray}
Recall that $G$ is a normally distributed r.v. with mean 0 and variance
1. Thus by the central limit theorem,
%
%e26 #&#
\begin{equation}
\label{poisontonormal} \lim_{n \rightarrow\infty} \pr( Z_n \geq n
) = \pr( G \geq B ).
\end{equation}
Note that for any fixed $t$, ${\mathcal A}(t) - {\mathcal D}(t) - B
\mu t$ is a nondegenerate Gaussian r.v., and every $x \in\reals$ is
a continuity point of the distribution of any nondegenerate Gaussian r.v.
It follows from Lemma \ref{renewalfclt3} and the definition of weak
convergence that for any fixed $t \geq0$ and all $x > 0$,
\[
\lim_{n \rightarrow\infty} \pr\Biggl( n^{-{1/2}} \Biggl(A_n(t)
- \sum_{i=1}^n N_i(t) \Biggr)
> x \Biggr) = \pr\bigl( {\mathcal A}(t) - {\mathcal D}(t) - B \mu t >
x \bigr).
\]
Thus for any fixed $x > 0$ and $s \geq0$,
%
%e27 #&#
\begin{eqnarray}\label{lowlower1}
&&
\liminf_{n \rightarrow\infty} \sup_{t \geq0} \pr\Biggl(
n^{-{1/2}} \Biggl(A_n(t) - \sum_{i=1}^n
N_i(t) \Biggr) > x \Biggr) \nonumber\\
&&\qquad\geq \liminf_{n \rightarrow\infty}
\pr\Biggl( n^{-{1/2}} \Biggl(A_n(s) - \sum
_{i=1}^n N_i(s) \Biggr) > x \Biggr)
\\
&&\qquad= \pr\bigl( {\mathcal A}(s) - {\mathcal D}(s) - B \mu s > x
\bigr).\nonumber
\end{eqnarray}
By fixing $x > 0$ and taking the supremum over all $s \geq0$ in (\ref
{lowlower1}),
we find that for all $x > 0$,
%
%e28 #&#
\begin{eqnarray}
\label{loweraa1}
&&
\liminf_{n \rightarrow\infty} \sup_{t \geq0}
\pr\Biggl( n^{-{1/2}} \Biggl(A_n(t) - \sum
_{i=1}^n N_i(t) \Biggr) > x \Biggr) \nonumber\\[-8pt]\\[-8pt]
&&\qquad\geq
\sup_{t \geq0} \pr\bigl( {\mathcal A}(t) - {\mathcal D}(t) - B \mu
t > x \bigr).\nonumber
\end{eqnarray}
Combining (\ref{lboundpleaseuse}), (\ref{poisontonormal}) and (\ref
{loweraa1}), we find that the left-hand side (l.h.s.) of (\ref
{lboundpleaseuse}) is at least
%
%e29 #&#
\begin{equation}
\label{lboundpleaseuse2} \pr( G \geq B ) \sup_{t \geq0} \pr\bigl(
{\mathcal A}(t) - {\mathcal D}(t) - B \mu t > x \bigr).
\end{equation}
Equation (\ref{lldd1}) then follows from (\ref{lboundpleaseuse2}) and
Corollary \ref{largedev3}(ii). Combining (\ref{lldd1}) with the
first part of Theorem \ref{ld1}, which we have already proven,
completes the proof.
\end{pf*}
%
%s7 #&#
\section{Application to Reed's weak limit}\label{appssec}
In \cite{R09}, Reed resolved the long-standing open question,
originally posed in \cite{HW81}, of the tightness and weak
convergence for the queue length of the transient $\mathit{GI}/\mathit{GI}/n$ queue in
the H--W regime. However, the associated weak limit is only described
implicitly, as the solution to a certain stochastic convolution
equation; see \cite{R09}.

In this section we derive bounds for the weak limit of the
transient $\mathit{GI}/\mathit{GI}/n$ queue in the H--W regime.
Let ${\mathcal Q}^n_1$ denote the FCFS $\mathit{GI}/\mathit{GI}/n$ queue with
inter-arrival times drawn i.i.d. distributed as $A \lambda_n^{-1}$,
processing times drawn i.i.d. distributed as $S$, and the following
initial conditions. For $i=1,\ldots,n$, there is a single job
initially being processed on server $i$, and the set of initial
processing times of these $n$ initial jobs is drawn i.i.d. distributed
as $R(S)$; there are zero jobs waiting in queue, and the first
inter-arrival time is distributed as $R(A \lambda_n^{-1})$,
independent of the initial processing times of those jobs initially in
system. Note that ${\mathcal Q}^n_1$ has the same initial conditions as
the FCFS $\mathit{GI}/\mathit{GI}/n$ queue ${\mathcal Q}$ we considered when constructing
our upper bound in Section \ref{uppersec}. Let $\hat{Q}_1(t)$ denote
the unique strong solution to~the stochastic convolution equation given
in \cite{R09}, equation 1.1, where we note that the initial conditions
and dynamics of $\hat{Q}_1$ have suitable interpretations as limits of
the initial conditions and dynamics of
${\mathcal Q}^n_1$ as $n \rightarrow\infty$. Then letting $Q^n_1(t)$
denote the number in system at time $t$ in ${\mathcal Q}^n_1$, in \cite
{R09}, the following is proven:
%
%th10 #&#
\begin{theorem}\label{reedtheorem}
For all $T \in(0,\infty)$, the sequence of stochastic processes
$\lbrace n^{-{1/2}}(Q^n_1(t) - n)^+_{0 \leq t \leq T}, n \geq1
\rbrace$ converges\vspace*{1pt} weakly to $\hat{Q}_1(t)_{0 \leq t \leq T}$
in the space $D[0,T]$ under the $J_1$ topology.
\end{theorem}
We now apply Theorem \ref{ubound1} to derive the first nontrivial
bounds for $\hat{Q}_1(t)$, proving:
%
%th11 #&#
\begin{theorem}\label{joshubound1}
For all $x > 0$ and $t \geq0$,
\[
\pr\bigl( \hat{Q}_1(t) > x \bigr) \leq\pr\Bigl( \sup
_{0 \leq s \leq t} \bigl( {\mathcal A}(s) - {\mathcal D}(s) - B \mu s
\bigr)
\geq x \Bigr).
\]
\end{theorem}
\begin{pf}
Note that we may let the arrival process to ${\mathcal Q}^n_1$ be
$A_n(t)$. Thus by Theorem \ref{ubound1}, for all $x > 0$ and $t \geq0$,
%
%e30 #&#
\begin{eqnarray}\label{jeq3}
&&
\liminf_{n \rightarrow\infty} \pr\bigl( n^{-{1/2}} \bigl(
\cal{Q}^n_1(t) - n \bigr)^+ > x \bigr) \nonumber\\
&&\qquad\leq \liminf
_{n \rightarrow\infty} \pr\Biggl( n^{-{1/2}} \sup_{0 \leq s \leq t}
\Biggl( A_n(s) - \sum_{i=1}^n
N_i(s) \Biggr) > x \Biggr)
\nonumber\\[-8pt]\\[-8pt]
&&\qquad\leq \limsup_{n \rightarrow\infty} \pr\Biggl( n^{-{1/2}} \sup
_{0 \leq s \leq t} \Biggl( A_n(s) - \sum
_{i=1}^n N_i(s) \Biggr) \geq x \Biggr)
\nonumber\\
&&\qquad\leq \pr\Bigl( \sup_{0 \leq s \leq t} \bigl( {\mathcal A}(s) -
{\mathcal
D}(s) - B \mu s \bigr) \geq x \Bigr)\nonumber
\end{eqnarray}
with the final inequality following from the Portmanteau theorem. Again
applying the Portmanteau theorem, it follows from Theorem
\ref{reedtheorem} that for all $x > 0$,
%
%e31 #&#
\begin{equation}
\label{jeq4} \pr\bigl( \hat{Q}_1(t) > x \bigr) \leq\liminf
_{n \rightarrow\infty} \pr\bigl( n^{-{1/2}} \bigl(Q^n_1(t)
- n \bigr)^+ > x \bigr).
\end{equation}
Combining (\ref{jeq3}) and (\ref{jeq4}) completes the proof.
\end{pf}
Theorem \ref{joshubound1} implies that $\hat{Q}_1(t)$ is
distributionally bounded over time, and thus is in a sense stable.
In particular, for all $t \geq0$, $\hat{Q}_1(t)$ is stochastically
dominated by the r.v. $\sup_{t \geq0} ( {\mathcal A}(t) -
{\mathcal D}(t) - B \mu t )$.

%s8 #&#
\section{Conclusion}\label{concsec}
In this paper, we studied the FCFS $\mathit{GI}/\mathit{GI}/n$ queue in the Halfin--Whitt
regime. We proved that under minor technical conditions the associated
sequence of steady-state queue length distributions, normalized by
$n^{{1/2}}$, is tight. We derived an upper bound for the large
deviation exponent of the limiting steady-state queue length matching
that conjectured in \cite{GM08}, and proved a matching lower bound
for the case of Poisson arrivals. We also derived the first nontrivial
bounds for the weak limit process studied in \cite{R09}. Our main
proof technique was the derivation of new and simple bounds for the
FCFS $\mathit{GI}/\mathit{GI}/n$ queue, which are of a structural nature, and exemplify a
general methodology which may be useful for analyzing a variety of
queueing systems.

This work leaves many interesting directions for future
research. One pressing question is whether or not $\lbrace n^{-{1/2}}
( Q^n(\infty) - n )^+, n \geq1 \rbrace$ has a
\textit{unique} weak limit, and thus converges weakly.
Indeed, such a result is only known for the cases of Markovian
processing times \cite{HW81}, deterministic processing times \cite
{JMM04}, and processing times with finite support \cite{GM08}. In
all of these cases,
either the distribution of $Q^n(\infty)$ can be computed explicitly
\cite{HW81,JMM04}, or can be represented as the steady-state
of a Markov chain whose dimension does not grow with $n$ \cite{GM08};
in the general setting, neither of these conditions hold. Similarly,
although Theorem~\ref{joshubound1} shows that the weak limit process
$\hat{Q}_1(t)$ is distributionally bounded over time, it is unknown
whether $\hat{Q}_1(t)$ has a well-defined stationary measure.
Furthermore, should $\lbrace n^{-{1/2}} ( Q^n(\infty) - n
)^+,\break n \geq1 \rbrace$ have a unique weak limit and $\hat
{Q}_1(t)$ have a well-defined stationary measure, must the two
coincide? We note that this question is intimately related to showing
that if one initializes ${\mathcal Q}^n$ with its stationary measure,
then the relevant sequence of scaled queueing processes converges weakly
(at the process level) to an appropriate stationary limit, and refer
the reader to
\cite{S63,HW81,PR00b,DHT10,RK10b} for progress along these lines. Similar questions (on the
order of fluid, as opposed to diffusion, scaling) were also
investigated in \cite{kang2012asymptotic}.

It would be interesting to extend our techniques to more
general models. For example, it should be possible to extend our lower
bounds to non-Poisson arrival processes, as was done in \cite{GM08}
for the special case of processing times with finite support. It would
also be interesting to analyze the large deviation behavior when the
finite second moment condition does not hold, since in this case the
large deviation exponent of Theorem \ref{ld1} equals zero, which
suggests that a fundamentally different qualitative behavior
may arise
in this setting. Finally, it would be interesting to generalize our
bounds to systems with abandonments ($\mathit{GI}/\mathit{GI}/n + \mathit{GI}$). This setting is
practically important, as the main application of the H--W regime has
been to the study of call-centers, for which customer abandonments are
an important modeling component \cite{AAM07}. For some interesting
steps along these lines the reader is referred to the recent papers
\cite{DH10,GS11a}.

\begin{appendix}\label{app}
%s9 #&#
\section*{Appendix}
%s9.1 #&#
\subsection{\texorpdfstring{Proof of Lemma \protect\ref{qom}}{Proof of Lemma 2}}
\mbox{}
\begin{pf}
Let $Z^i(t)$ denote the number of jobs initially in ${\mathcal Q}^i$
which are still in ${\mathcal Q}^i$ at time $t$, $i \in\lbrace1,2
\rbrace$.
We claim that $Z^2(t) \geq Z^1(t)$ for all $t \geq0$. Indeed, let $J$
be any job initially in ${\mathcal Q}^1$, and let $S_J$ denote its
initial processing time. Then \ref{as2} ensures the existence of a
distinct corresponding job $J'$ initially in ${\mathcal Q}^2$, with the
same initial processing time $S_J$.
Since by \ref{as1} all jobs initially in $\cal{Q}^1$ begin
processing at time $0$, it follows that $J$ departs ${\mathcal Q}^1$ at
time~$S_J$, while $J'$ departs ${\mathcal Q}^2$ no earlier than $S_J$.
Making this argument for each job $J$ initially in ${\mathcal Q}^1$
proves that $Z^2(t) \geq Z^1(t)$ for all $t \geq0$.

Let $D^i_k$ denote the time at which the job that arrives to
${\mathcal Q}^i$ at time $T^1_k$ departs from ${\mathcal Q}^i$, $k \geq
1, i \in\lbrace1,2 \rbrace$.
We now prove by induction that for $k \geq1$, $D^2_k \geq D^1_k$, from
which the proposition follows. Observe that for all $k \geq1$,
%
%e32 #&#
\begin{equation}
\label{d11eq1} D^1_k = \inf\Biggl\lbrace t\dvtx  t \geq
T^1_k, Z^1(t) + \sum
_{j=1}^{k-1} I\bigl( D^1_j > t
\bigr) \leq n-1 \Biggr\rbrace+ S^1_k.
\end{equation}
Also,
%
%e33 #&#
\begin{equation}
\label{d22eq1} D^2_k \geq\inf\Biggl\lbrace t\dvtx  t \geq
T^1_k, Z^1(t) + \sum
_{j=1}^{k-1} I\bigl( D^2_j > t
\bigr) \leq n-1 \Biggr\rbrace+ S^1_k,
\end{equation}
where the inequality in (\ref{d22eq1}) arises since $Z^2(t) \geq
Z^1(t)$ for all $t \geq0$, and the job that arrives to ${\mathcal
Q}^2$ at time $T^1_k$ may have to wait for additional jobs, which
either were initially present in ${\mathcal Q}^2$ but not ${\mathcal
Q}^1$, or which arrive at a time belonging to $\lbrace T^2_k, k \geq1
\rbrace\setminus\lbrace T^1_k, k \geq1 \rbrace$.

For the base case $k = 1$, note that
$D^1_1 = \inf\lbrace t\dvtx  t \geq T^1_1, Z^1(t) \leq n-1 \rbrace+
S^1_1$, while
$D^2_1 \geq\inf\lbrace t\dvtx  t \geq T^1_1, Z^1(t) \leq n-1 \rbrace+
S^1_1$.

Now assume the induction is true for all $j \leq k$. Then for
all $t \geq0$,\break $\sum_{j=1}^{k} I( D^2_j > t ) \geq\sum_{j=1}^{k} I(
D^1_j > t )$.
Thus
\begin{eqnarray*}
&&\inf\Biggl\lbrace t\dvtx  t \geq T^1_{k+1}, Z^1(t) +
\sum_{j=1}^{k} I\bigl( D^1_j
> t \bigr) \leq n-1 \Biggr\rbrace+ S^1_{k+1}
\\
&&\qquad\leq\inf\Biggl\lbrace t\dvtx  t \geq T^1_{k+1},
Z^1(t) + \sum_{j=1}^{k} I\bigl(
D^2_j > t \bigr) \leq n-1 \Biggr\rbrace+
S^1_{k+1}.
\end{eqnarray*}
It then follows from (\ref{d11eq1}) and (\ref{d22eq1}) that
$D^1_{k+1} \leq D^2_{k+1}$, completing the induction.
\end{pf}

%s9.2 #&#
\subsection{\texorpdfstring{Proof of Theorem \protect\ref{statwait2}}{Proof of Theorem 5}}
In \cite{Stight99}, Theorem 1 (given in the notation of~\cite
{Stight99}), the following is proven:
%
%th12 #&#
\begin{theorem}\label{statwait1}
Suppose that for all sufficiently large $n$, $\lbrace\zeta_{n,i}, i
\geq1 \rbrace$ is a stationary, countably infinite sequence of r.v.
Let $a_n \stackrel{\Delta}{=} \E[ \zeta_{n,1} ]$, and
$W_{n,k} \stackrel{\Delta}{=} \sum_{i=1}^k \zeta_{n,i}$. Further
assume that $a_n < 0, \lim_{n \rightarrow\infty} a_n = 0$, and there
exist $C_1,\break C_2 < \infty$ and $\varepsilon> 0$ s.t. for all sufficiently
large $n$:
\begin{longlist}
\item$\E[ | W_{n,k} - k a_n |^{2 + \varepsilon} ] \leq C_1
k^{1 + {\varepsilon}/{2}}$ for all $k \geq1$;
\item$\pr( \max_{i=1,\ldots,k} ( W_{n,i} - i a_n ) > x )
\leq C_2 \E[ |W_{n,k} - k a_n |^{2 + \varepsilon} ] x^{- (2 +
\varepsilon)}$ for all $k \geq1$ and $x > 0$;
\item$\pr( \lim_{k \rightarrow\infty} W_{n,k} = -\infty) = 1$.
\end{longlist}
Then $\lbrace|a_n| \sup_{k \geq0} W_{n,k}, n \geq1 \rbrace$ is tight.
\end{theorem}
With Theorem \ref{statwait1} in hand, we now complete the proof of
Theorem \ref{statwait2}.
\begin{pf*}{Proof of Theorem \ref{statwait2}}
The proof follows almost exactly as the proof of Theorem \ref
{statwait1} given in \cite{Stight99}, and we now explicitly comment on
precisely where the proof must be changed superficially so as to carry
through under the slightly different set of assumptions of Theorem
\ref{statwait2}. First off, nowhere in the proof of
Theorem~\ref{statwait1} given in \cite{Stight99} is assumption (iii) of
Theorem \ref{statwait1} used, and thus that assumption is extraneous
and may be removed. The only other difference between the set of
assumptions for Theorem \ref{statwait1} and the set of assumptions for
Theorem \ref{statwait2} is that assumption (ii) of Theorem \ref
{statwait1} is replaced by assumption (ii) of Theorem \ref {statwait2}.
We now show that Theorem \ref{statwait1} holds under this change in
assumptions. As in \cite{Stight99}, let $x(a_n,k) \stackrel{\Delta}{=}
\frac{x}{|a_n|} + 2^k |a_n|$. Then the\vspace*{1pt} only place where assumption (ii)
of Theorem \ref{statwait1} is used is between equations 5 and 6, where
this assumption is required to demonstrate that
%
%e34 #&#
%e35 #&#
\begin{eqnarray}
\label{oldass1}
&& \pr\biggl( W_{n,2^k} - 2^k a_n >
\frac{1}{2}x(a_n,k) \biggr) \nonumber\\[-8pt]\\[-8pt]
&&\quad{}+ \pr\Biggl( \max
_{i=0,\ldots,2^k} \Biggl( \sum_{j=1}^i
\zeta_{n, j + 2^k} - i a_n \Biggr) > \frac{1}{2}x(a_n,k)
\Biggr)
\nonumber\\
\label{oldassineq}
&&\qquad\leq(1 + C_2) C_1 2^{2 + \varepsilon} 2^{k(1 + {\varepsilon}/{2})}
\bigl( x(a_n,k) \bigr)^{-(2 + \varepsilon)}.%\
\end{eqnarray}
We now prove that assumption (ii) of Theorem \ref{statwait2} is
sufficient to derive (\ref{oldassineq}). In particular, the first
summand of (\ref{oldass1}) is at most
%
%e36 #&#
\begin{eqnarray}\label{newnewass1}
&& \E\bigl[\bigl|W_{n,2^k} - 2^k a_n
\bigr|^{2 + \varepsilon} \bigr] \bigl( \tfrac{1}{2}x(a_n,k)
\bigr)^{- (2 + \varepsilon) } \qquad\mbox{by Markov's inequality}
\nonumber\\
&&\qquad\leq C_1 2^{2 + \varepsilon} 2^{k(1 + {\varepsilon}/{2})} \bigl(
x(a_n,k) \bigr)^{- (2 + \varepsilon) } \\
&&\quad\qquad\mbox{by assumption (i) of Theorem
\ref{statwait2}.}\hspace*{-25pt}\nonumber
\end{eqnarray}
By the stationarity of $\lbrace\zeta_{n,i}, i \geq1 \rbrace$, the
second summand of (\ref{oldass1}) equals
%
%e37 #&#
\begin{eqnarray}\label{newnewass2}
&& \pr\biggl( \max_{i=0,\ldots,2^k} ( W_{n,i} - i
a_n ) > \frac
{1}{2}x(a_n,k) \biggr)
\nonumber\\
&&\qquad\leq C_2 2^{2 + \varepsilon} 2^{k(1 + {\varepsilon}/{2})} \bigl(
x(a_n,k) \bigr)^{- (2 + \varepsilon) } \\
&&\qquad\quad\mbox{by assumption (ii) of Theorem
\ref{statwait2}}.\nonumber
\end{eqnarray}
Since we may w.l.o.g. take $C_1,C_2 \geq1$, it follows that $C_1 + C_2
\leq(1 + C_2) C_1$, and thus (\ref{oldassineq}) follows from (\ref
{newnewass1}) and (\ref{newnewass2}). The theorem follows from the
proof of Theorem \ref{statwait1} given in \cite{Stight99}.
\end{pf*}

%s9.3 #&#
\subsection{\texorpdfstring{Proof of Lemma \protect\ref{binomial2}}{Proof of Lemma 5}}
We note that the special case $r = 2$ is treated in~\cite{Whitt85}.
Before proceeding with the proof of Lemma \ref{binomial2}, it will be
useful to prove three auxiliary results. The first treats the special
case $n = 1, t \geq1$ for ordinary (as opposed to equilibrium) renewal
processes, and is proven in Theorem 1 of \cite{CHL79}.
%
%th13 #&#
\begin{theorem}\label{extendedrenewalCLT200}
Suppose $Z(t)$ is an ordinary renewal process with renewal distribution
$X$ s.t. $\E[X] = \mu^{-1} \in(0, \infty)$, and $\E[X^r] < \infty
$ for some $r \geq2$. Then
$\sup_{t \geq1} t^{- {r}/{2} } \E[| Z(t) - \mu t |^r
] < \infty$.
\end{theorem}
Second, we prove a lemma treating the special case $n = 1, t \geq1$
for equilibrium renewal processes.
%
%le8 #&#
\begin{lemma}\label{extendedrenewalCLT2}
Under the same definitions and assumptions as Lemma \ref{binomial2},
for each $r \geq2$, there exists $C_{X,r} < \infty$ (depending only
on $X$ and $r$) s.t.
for all $t \geq1$, $\E[| Z^e_1(t) - \mu t |^r ] < C_{X,r}
t^{{r/2}}$.
\end{lemma}
\begin{pf}
Let $X^e$ denote the first renewal interval in $Z^e_1(t)$, and
$f_{X^e}$ its density function, whose existence is guaranteed by (\ref
{rezz}). Observe that we may construct $Z^e_1(t)$ and an ordinary
renewal process $Z(t)$ (also with renewal distribution $X$) on the same
probability space so that
for all $t \geq0$, $Z^e_1(t) = I( X^e \leq t) + Z ( ( t - X^e)^+
)$, with $Z(t)$ independent of $X^e$.
Thus
\[
Z^e_1(t) - \mu t = \bigl( Z \bigl( \bigl( t -
X^e\bigr)^+ \bigr) - \mu\bigl( t - X^e \bigr)^+ \bigr)
+ \bigl( I \bigl( X^e \leq t \bigr) - \mu\bigl(t - \bigl(t -
X^e\bigr)^+ \bigr) \bigr).
\]
Fixing some $t \geq1$, it follows that $\E[| Z^e_1(t) - \mu t
|^r ]$ is at most
%
%e38 #&#
%e39 #&#
\begin{eqnarray}
\label{verylarge2}
&& 2^{r - 1} \E\bigl[\bigl| Z \bigl( \bigl( t - X^e\bigr)^+
\bigr) - \mu\bigl( t - X^e \bigr)^+\bigr|^r \bigr]
\\
\label{verylarge3}
&&\qquad{} + 2^{r-1} \E\bigl[ \bigl| I \bigl( X^e \leq t \bigr) - \mu
\bigl(t - \bigl(t - X^e\bigr)^+ \bigr) \bigr|^r \bigr]\nonumber\\[-8pt]\\[-8pt]
&&\quad\qquad\mbox{by the triangle inequality and
(\ref{convexlemma}).}\nonumber
\end{eqnarray}
We now bound the term $\E[| Z ( ( t - X^e)^+ ) - \mu
( t - X^e )^+|^r ]$ appearing in (\ref{verylarge2}),
which equals
%
%e40 #&#
\begin{eqnarray}\label{verylarge2b}
&&\int_0^{t-1} \E\bigl[\bigl| Z ( t - s ) - \mu( t
- s ) \bigr|^r \bigr] f_{X^e}(s) \,ds \nonumber\\[-8pt]\\[-8pt]
&&\qquad{}+ \int_{t-1}^{t}
\E\bigl[\bigl| Z ( t - s ) - \mu( t - s ) \bigr|^r \bigr]
f_{X^e}(s) \,ds.\nonumber
\end{eqnarray}
Let $C'_{X,r} \stackrel{\Delta}{=} \sup_{t \geq1} t^{- {r}/{2}
} \E[| Z(t) - \mu t |^r ]$. Theorem \ref
{extendedrenewalCLT200} implies that
the first summand of (\ref{verylarge2b}) is at most
\begin{eqnarray*}
\int_0^{t-1} \bigl(C'_{X,r}
(t-s)^{{r/2}} \bigr) f_{X^e}(s) \,ds &\leq& \int
_0^{t-1} \bigl(C'_{X,r}
t^{{r/2}} \bigr) f_{X^e}(s) \,ds
\\
&=& C'_{X,r} t^{{r/2}} \pr\bigl( X^e
\leq t-1 \bigr).
\end{eqnarray*}
Since $t - s \leq1$ implies $|Z(t-s) - \mu(t-s)|^r \leq|Z(1) + \mu
|^r$, the second summand of (\ref{verylarge2b}) is at most
$\E[| Z( 1 ) + \mu|^r ] \pr( X^e \in[t-1,t] )$.
Combining our bounds for (\ref{verylarge2b}), we find that (\ref
{verylarge2}) is at most
%
%e41 #&#
\begin{equation}
\label{verylarge2b2} 2^{r-1} \E\bigl[\bigl| Z( 1 ) + \mu\bigr|^r
\bigr] + 2^{r-1} C'_{X,r} t^{{r/2}}.
\end{equation}
We now bound (\ref{verylarge3}), which is at most
%
%e42 #&#
\begin{eqnarray}\label{verylarge33a}
&&
2^{r-1} \E\bigl[ \bigl| I \bigl( X^e \leq t \bigr) + \mu
\bigl(t - \bigl(t - X^e\bigr)^+ \bigr) \bigr|^r \bigr] \nonumber\\
&&\qquad\leq
2^{2r - 2} \bigl(1 + \E\bigl[ \bigl|\mu\bigl(t - \bigl(t - X^e
\bigr)^+ \bigr)\bigr|^r \bigr] \bigr) \qquad\mbox{by (\ref{convexlemma})}
\\
&&\qquad= 2^{2r - 2} \biggl(1 + \mu^r \biggl( \int
_0^t s^r f_{X^e}(s) \,ds +
\int_t^{\infty} t^r
f_{X^e}(s) \,ds \biggr) \biggr).\nonumber
\end{eqnarray}
It follows from (\ref{rezz}) and Markov's inequality that for all $s
\geq0$,
$f_{X^e}(s) = \mu\pr(X > s) \leq\mu\E[X^r] s^{-r}$. Thus the term
$\int_0^t s^r f_{X^e}(s) \,ds + \int_t^{\infty} t^r f_{X^e}(s) \,ds$
appearing in (\ref{verylarge33a}) is at most
%
%e43 #&#
\begin{eqnarray}\label{verylarge3b3}
&&
\int_0^t s^r \bigl( \mu\E
\bigl[X^r\bigr] s^{-r} \bigr) \,ds + t^r \int
_t^{\infty} \bigl( \mu\E\bigl[X^r\bigr]
s^{-r} \bigr) \,ds \nonumber\\
&&\qquad= \mu\E\bigl[X^r\bigr] \biggl( \int
_0^t ds + t^r \int
_t^{\infty} s^{-r} \,ds \biggr)
\nonumber\\[-8pt]\\[-8pt]
&&\qquad= \mu\E\bigl[X^r\bigr] \bigl( t + t^r
(r-1)^{-1} t^{1-r} \bigr)
\nonumber\\
&&\qquad= \mu\E\bigl[X^r\bigr] \bigl( 1 + (r-1)^{-1} \bigr) t.\nonumber
\end{eqnarray}
Using (\ref{verylarge2b2}) to bound (\ref{verylarge2}) and (\ref
{verylarge3b3}) to bound (\ref{verylarge33a}) and (\ref{verylarge3}),
we find that $\E[| Z^e_1(t) - \mu t |^r ]$ is at most
%
%e44 #&#
\begin{eqnarray}
\label{enditbig}
&&2^{r-1} \E\bigl[\bigl| Z( 1 ) + \mu\bigr|^r \bigr]
+ 2^{r-1} C'_{X,r} t^{{r/2}} +
2^{2r - 2} \nonumber\\[-8pt]\\[-8pt]
&&\qquad{}+ 2^{2r - 2} \mu^{r+1} \E\bigl[X^r
\bigr] \bigl( 1 + (r-1)^{-1} \bigr) t.\nonumber
\end{eqnarray}
Noting that $\E[| Z( 1 ) + \mu|^r ] < \infty$ since any
renewal process, evaluated at any fixed time, has finite moments of all
orders (see \cite{P65}, page 155), $\E[X^r] < \infty$ by assumption,
and $t \leq t^{{r/2}}$ since $t \geq1$ and $\frac{r}{2} \geq
1$, the lemma follows from~(\ref{enditbig}).
\end{pf}
Third, we prove a lemma which will be useful in handling the case $t
\leq2$. We note that in this auxiliary lemma, the upper bound is of
the form $(nt)^r$, as opposed to $(nt)^{{r/2}}$.
%
%le9 #&#
\begin{lemma}\label{binomial1}
Under the same definitions and assumptions as Lemma \ref{binomial2},
there exists $C_{X,r} < \infty$ (depending only on $X$ and $r$) s.t.
for all $n \geq1$, and $t \in[0, 2 ]$,
%
%e45 #&#
\begin{equation}
\label{binomial1b}\E\Biggl[ \Biggl|\sum_{i=1}^n
Z^e_i(t) - \mu n t \Biggr|^r \Biggr] \leq
C_{X,r} \bigl( 1 + ( n t )^r \bigr).
\end{equation}
\end{lemma}
\begin{pf}
Note that the left-hand side of (\ref{binomial1b}) is at most
%
%e46 #&#
\begin{equation}\label{referenceeq1}
\E\Biggl[\Biggl|\sum_{i=1}^n
Z^e_i(t) + \mu n t \Biggr|^r \Biggr] \leq
2^{r - 1} \Biggl( \E\Biggl[ \Biggl(\sum_{i=1}^n
Z^e_i(t) \Biggr)^r \Biggr] + (\mu
nt)^r \Biggr) \qquad\mbox{by (\ref{convexlemma})}.\hspace*{-28pt}
\end{equation}
We now bound the term $\E[ (\sum_{i=1}^n Z^e_i(t)
)^r ]$ appearing in (\ref{referenceeq1}). Let $\lbrace Z_i(t)
\rbrace$ denote a countably infinite sequence of i.i.d. ordinary
renewal processes with renewal distribution $X$. Let us fix some $t \in
[0,2]$ and $n \geq1$, and let $\lbrace B_i \rbrace$ denote a
countably infinite sequence of i.i.d. Bernoulli r.v. s.t. $\pr(B_i = 1)
= p \stackrel{\Delta}{=} \pr( R(X) \leq t)$.
Note that we may construct $\lbrace Z^e_i(t) \rbrace$, $\lbrace Z_i(t)
\rbrace, \lbrace B_i \rbrace$ on the same probability space so that
w.p.1 $Z^e_i(t) \leq B_i (1 + Z_i(t) )$ for all $i \geq1$,
with $\lbrace Z_i(t) \rbrace, \lbrace B_i \rbrace$ mutually independent.
Letting $M \stackrel{\Delta}{=} \sum_{i=1}^n B_i$, it follows that
%
%e47 #&#
\begin{equation}
\label{appeq1} \E\Biggl[ \Biggl( \sum_{i=1}^n
Z^e_i(t) \Biggr)^{\lceil r \rceil} \Biggr] \leq\E\Biggl[
\Biggl( \sum_{i=1}^{ M } \bigl(1 +
Z_i(t) \bigr) \Biggr)^{\lceil r \rceil} \Biggr].
\end{equation}
Let $Z^+$ denote the set of nonnegative integers. Note that for any
positive integer $k$,
%
%e48 #&#
\begin{eqnarray}\label{appeq2}
\E\Biggl[ \Biggl( \sum_{i=1}^k \bigl( 1
+ Z_i(t) \bigr) \Biggr)^{\lceil r \rceil} \Biggr] &=& \E\Biggl[ \sum
_{ \stackrel{j_1,\ldots,j_k \in Z^+}{j_1 + \cdots+
j_k = \lceil r \rceil}} \prod_{i=1}^k
\bigl( 1 + Z_i(t) \bigr)^{j_i} \Biggr]
\nonumber\\
&=& \sum_{\stackrel{j_1,\ldots,j_k \in Z^+}{j_1 + \cdots+ j_k =
\lceil r \rceil}} \prod_{i=1}^k
\E\bigl[ \bigl( 1 + Z_i(t) \bigr)^{j_i} \bigr] \\
&&\mbox{since }
\bigl\lbrace Z_i(t) \bigr\rbrace\mbox{ are i.i.d. r.v.s.}\nonumber
\end{eqnarray}
For any setting of $\lbrace j_i, i=1,\ldots,k \rbrace$ in the
right-hand side of (\ref{appeq2}), at most $\lceil r \rceil$ of the
$j_i$ are strictly positive, and
each $j_i$ is at most $\lceil r \rceil$. It follows that
the term $\prod_{i=1}^k \E[ ( 1 + Z_i(t) )^{j_i}
]$ appearing in the right-hand side of (\ref{appeq2}) is at most
$ ( \E[ ( 1 + Z_1(t) )^{\lceil r \rceil} ]
)^{\lceil r \rceil}$, regardless of the particular setting of
$\lbrace j_i, i=1,\ldots,k \rbrace$.
As there are a total of $k^{\lceil r \rceil}$ distinct feasible
configurations for $\lbrace j_i, i=1,\ldots,k \rbrace$ in the
right-hand side of (\ref{appeq2}), combining the above we find that
for any nonnegative integer $k$,
%
%e49 #&#
\begin{eqnarray}\label{appeq3}
\E\Biggl[ \Biggl( \sum_{i=1}^k \bigl( 1
+ Z_i(t) \bigr) \Biggr)^{\lceil r \rceil} \Biggr] &\leq&
k^{\lceil r \rceil} \bigl( \E\bigl[ \bigl( 1 + Z_1(t)
\bigr)^{\lceil
r \rceil} \bigr] \bigr)^{\lceil r \rceil}
\nonumber\\
&\leq& k^{\lceil r \rceil} \bigl( \E\bigl[ \bigl( 1 + Z_1(2)
\bigr)^{\lceil r \rceil} \bigr] \bigr)^{\lceil r \rceil} \\
&&\mbox{since by
assumption } t
\leq2.\nonumber
\end{eqnarray}
Since any renewal process, evaluated at any fixed time, has finite
moments of all orders (see \cite{P65}, page 155), it follows that
$C^1_{X,\lceil r \rceil} \stackrel{\Delta}{=} ( \E[
( 1 + Z_1(2) )^{\lceil r \rceil} ] )^{\lceil r \rceil
}$ is a finite constant depending only on $X$ and $\lceil r \rceil$.
Combining (\ref{appeq1}) and (\ref{appeq3}) with the independence of
$M$ and $\lbrace Z_i(t) \rbrace$, it follows from a simple
conditioning argument that
%
%e50 #&#
\begin{equation}
\label{appeq4} \E\Biggl[ \Biggl( \sum_{i=1}^n
Z^e_i(t) \Biggr)^{\lceil r \rceil} \Biggr] \leq
C^1_{X,\lceil r \rceil} \E\bigl[M^{\lceil r \rceil} \bigr].
\end{equation}
We now bound the term $\E[M^{\lceil r \rceil} ]$ appearing
in (\ref{appeq4}). Noting that $M$ is a binomial distribution with
parameters $n$ and $p$, it follows from \cite{R37}, equation 3.3, that
there exist finite constants $C_{0,\lceil r \rceil},C_{1,\lceil r
\rceil},C_{2,\lceil r \rceil},\ldots, C_{\lceil r \rceil,\lceil r
\rceil}$, independent of $n$ and $p$, s.t. $\E[M^{\lceil r
\rceil} ] = \sum_{k=0}^{\lceil r \rceil} C_{k,\lceil r \rceil}
p^k \prod_{j=0}^{k-1} (n - j)$. Further noting that $\prod_{j=0}^{k-1}
(n - j) \leq n^k$ for all $k \geq0$, it follows that $\E
[M^{\lceil r \rceil} ] \leq\sum_{k=0}^{\lceil r \rceil}
|C_{k,\lceil r \rceil}| (n p)^k$.
Letting $C^2_{\lceil r \rceil} \stackrel{\Delta}{=} \max_{i=0,\ldots,\lceil r \rceil} | C_{i,\lceil r \rceil} |$, it follows
from (\ref{appeq4}) that
%
%e51 #&#
\begin{eqnarray}\label{appeq6}
\E\Biggl[ \Biggl( \sum_{i=1}^n
Z^e_i(t) \Biggr)^{\lceil r \rceil} \Biggr] &\leq&
C^1_{X, \lceil r \rceil} C^2_{\lceil r \rceil} \sum
_{i=0}^{\lceil r \rceil} (np)^i
\nonumber\\[-8pt]\\[-8pt]
&\leq& C^1_{X, \lceil r \rceil} C^2_{\lceil r \rceil} \bigl(
\lceil r \rceil+ 1\bigr) (1 + np )^{\lceil r \rceil}.\nonumber
\end{eqnarray}
Recall that for any nonnegative r.v. $Y$, one has that $\E[Y^r] \leq
\E[Y^{\lceil r \rceil}]^{{r/\lceil r \rceil}}$. Thus letting
$C^3_{X,r} \stackrel{\Delta}{=} ( C^1_{X, \lceil r \rceil}
C^2_{\lceil r \rceil}
(\lceil r \rceil+ 1) )^{ {r}/{\lceil r \rceil} }$,
it follows from (\ref{appeq6}) that
%
%e52 #&#
\begin{equation}
\label{appeq10} \E\Biggl[ \Biggl( \sum_{i=1}^n
Z^e_i(t) \Biggr)^r \Biggr] \leq
C^3_{X,r}(1 + np)^r.
\end{equation}
Furthermore, it follows from (\ref{rezz}) that $p = \mu\int_0^t \pr
(X > y) \,dy \leq\mu t$.
Combining with (\ref{appeq10}), we find that
%
%e53 #&#
\begin{equation}
\label{appeq11} \E\Biggl[ \Biggl( \sum_{i=1}^n
Z^e_i(t) \Biggr)^r \Biggr] \leq
C^3_{X,r} (1 + \mu n t)^r.
\end{equation}
Plugging (\ref{appeq11}) back into (\ref{referenceeq1}), it follows
that the left-hand side of (\ref{binomial1b}) is at most
\[
2^{r - 1} \bigl( C^3_{X,r} (1 + \mu n
t)^r + (\mu nt)^r \bigr) \leq2^{r}
\bigl(C^3_{X,r}+1\bigr) (1 + \mu n t)^r.
\]
Noting that $(1 + \mu n t)^r \leq2^r ( 1 + (\mu n t)^r )$ by
(\ref{convexlemma}), and
$1 + (\mu n t)^r \leq(1 + \mu)^r (1 + (nt)^r )$, completes
the proof.
\end{pf}
With the above auxiliary results in hand, we now complete the proof of
Lem\-ma~\ref{binomial2}.
\begin{pf*}{Proof of Lemma \ref{binomial2}}
We proceed by a case analysis. First, suppose $t \leq\frac{2}{n}$.
Then we also have $t \leq2$, and by Lemma \ref{binomial1} there
exists $C^1_{X,r} < \infty$ s.t. the left-hand side of (\ref{focus1})
is at most
\[
C^1_{X,r} \bigl( 1 + ( n t )^r \bigr) \leq
C^1_{X,r}\bigl( 1 + 2^r \bigr) \qquad\mbox{since } t
\leq\frac{2}{n} \mbox{ implies } n t \leq2.
\]
Letting $M_1 \stackrel{\Delta}{=} C^1_{X,r}( 1 + 2^r )$, it follows that
the left-hand side of (\ref{focus1}) is at most
$M_1 \leq M_1 (1 + (n t)^{{r/2}} )$, completing the
proof for the case $t \leq\frac{2}{n}$.

Second, suppose $t \in[ \frac{2}{n}, 2 ]$. Let
$n_1(t) \stackrel{\Delta}{=} \lfloor n t \rfloor$. Noting that $t
\geq\frac{2}{n}$ implies \mbox{$n_1(t) > 0$}, in this case we may define
$n_2(t) \stackrel{\Delta}{=} \lfloor\frac{n}{n_1(t)} \rfloor$.
Then the left-hand side of (\ref{focus1}) equals
%
%e54 #&#
%e55 #&#
\begin{eqnarray}
&& \E\Biggl[ \Biggl| \sum_{m = 1}^{n_1(t)} \sum
_{l=1}^{ n_2(t) } \bigl( Z^e_{(m-1)n_2(t) + l }
( t ) - \mu t \bigr) + \sum_{l = n_1(t) n_2(t)
+ 1}^n
\bigl( Z^e_l ( t ) - \mu t \bigr) \Biggr|^{r}
\Biggr]
\nonumber
\\
\label{maxme1}
&&\qquad\leq 2^{r-1} \E\Biggl[ \Biggl| \sum_{m = 1}^{n_1(t)}
\sum_{l=1}^{
n_2(t)} \bigl( Z^e_{ (m-1)n_2(t) + l }
( t ) - \mu t \bigr) \Biggr|^r \Biggr]
\\
\label{maxme2}
&&\qquad\quad{} + 2^{r-1} \E\Biggl[ \Biggl| \sum_{l = n_1(t) n_2(t) + 1}^n
\bigl( Z^e_l ( t ) - \mu t \bigr) \Biggr|^{r}
\Biggr] \qquad\mbox{by the tri. ineq. and (\ref{convexlemma}).}
\end{eqnarray}
We now bound (\ref{maxme1}). By Lemma \ref{csumbound}, there exists
$C_{r}<\infty$ s.t. (\ref{maxme1}) is at most
%
%e56 #&#
\begin{eqnarray}\label{abcd1}
&& 2^{r-1} C_{r} \bigl(n_1(t)
\bigr)^{{r/2}} \E\Biggl[ \Biggl| \sum_{l=1}^{ n_2(t) }
\bigl( Z^e_{l} ( t ) - \mu t \bigr) \Biggr|^r
\Biggr]
\nonumber\\[-8pt]\\[-8pt]
&&\qquad\leq 2^{r-1} C_{r} \bigl(n_1(t)
\bigr)^{{r/2}} \bigl(C^1_{X,r} \bigl( 1 + \bigl(
n_2(t) t \bigr)^r \bigr) \bigr) \qquad\mbox{by Lemma
\ref{binomial1}, since } t \leq2.\hspace*{-18pt}\nonumber
\end{eqnarray}
We now bound the term $t n_2(t)$ appearing in (\ref{abcd1}). In particular,
%
%e57 #&#
\begin{equation}
\label{abcd2} t n_2(t) = t \biggl\lfloor\frac{n}{ \lfloor n t \rfloor}
\biggr
\rfloor\leq\frac{nt}{nt - 1}.
\end{equation}
But since $t \geq\frac{2}{n}$ implies $nt \geq2$, and $g(z)
\stackrel{\Delta}{=}\frac{z}{z-1}$ is a decreasing function of $z$
on $(1,\infty)$, it follows from (\ref{abcd2}) that
\[
t n_2(t) \leq2.
\]
Since $n_1(t) \leq nt$, it thus follows from (\ref{abcd1}) that (\ref
{maxme1}) is at most
%
%e58 #&#
\begin{equation}
\label{binbin3a} 2^{r-1} C_{r} C^1_{X,r}
\bigl(1 + 2^r\bigr) (nt)^{{r/2}}.
\end{equation}
We now bound (\ref{maxme2}). Note that the sum $\sum_{l = n_1(t)
n_2(t) + 1}^n ( Z^e_l ( t ) - \mu t )$ appearing in (\ref
{maxme2}) is taken over $n - n_1(t) n_2(t)$ terms. Furthermore,
\begin{eqnarray*}
n - n_1(t) n_2(t) &=& n - n_1(t) \biggl
\lfloor\frac{n}{n_1(t)} \biggr\rfloor
\\
&\leq& n - n_1(t) \biggl( \frac{n}{n_1(t)} - 1 \biggr)
\\
&=& n_1(t).
\end{eqnarray*}
As $n_1(t) \leq nt$, it thus follows from Lemma \ref{csumbound} that
(\ref{maxme2}) is at most
%
%e59 #&#
\begin{eqnarray}\label{abcd6}
&&
2^{r-1} C_r (nt)^{{r/2}} \E\bigl[ \bigl|
Z^e_1(t) - \mu t \bigr|^r \bigr] \nonumber\\
&&\qquad\leq
2^{r-1} C_r (nt)^{{r/2}} \E\bigl[ \bigl(
Z^e_1(t) + \mu t \bigr)^r \bigr]
\\
&&\qquad\leq 2^{r-1} C_r (nt)^{{r/2}} \E\bigl[ \bigl(
Z^e_1(2) + 2 \mu\bigr)^r \bigr]
\qquad\mbox{since } t \leq2.\nonumber
\end{eqnarray}
Using (\ref{binbin3a}) to bound (\ref{maxme1}) and (\ref{abcd6}) to
bound (\ref{maxme2}) shows that the left-hand side of (\ref{focus1})
is at most
%
%e60 #&#
\begin{equation}
\label{binbin3} 2^{r-1} C_{r} C^1_{X,r}
\bigl(1 + 2^r\bigr) (nt)^{{r/2}} + 2^{r-1}
C_r (nt)^{{r/2}} \E\bigl[ \bigl( Z^e_1(2)
+ 2 \mu\bigr)^r \bigr].
\end{equation}
Let $M_2 \stackrel{\Delta}{=} 2^{r-1} C_{r} C^1_{X,r} (1 + 2^r) +
2^{r-1} C_r \E[ ( Z^e_1(2) + 2 \mu)^r ]$.
It follows from (\ref{binbin3}) that
the left-hand side of (\ref{focus1}) is at most $M_2 (nt)^{
{r/2}} \leq M_2 (1 + (nt)^{{r/2}} )$, completing the
proof for the case $t \in[ \frac{2}{n}, 2 ]$.

Finally, suppose $t \geq2$. In this case, it follows from
Lemma \ref{csumbound} that
the l.h.s. of (\ref{focus1}) is at most $C_r n^{{r/2}} \E
[|Z^e_1(t) - \mu t|^r ]$. Let $C^2_{X,r} \stackrel{\Delta}{=}
\sup_{t \geq2} t^{-{r/2}} \E[|Z^e_1(t) - \mu t|^r
]$. Then it follows from Lemma \ref{extendedrenewalCLT2} that
$C^2_{X,r} < \infty$, and the left-hand side of (\ref{focus1}) is at most
$C_r C^2_{X,r} (nt)^{{r/2}}$. Letting $M_3 \stackrel{\Delta}{=}
C_r C^2_{X,r}$, it follows that
the left-hand side of (\ref{focus1}) is at most $M_3 (nt)^{
{r/2}} \leq M_3 (1 + (nt)^{{r/2}} )$, completing the
proof for the case $t \geq2$.

As this treats all cases, we can complete the proof of the
lemma by letting $M_4 \stackrel{\Delta}{=} \max(
M_1,M_2,M_3 )$, and noting that for all $n \geq1$ and $t \geq0$,
the left-hand side of (\ref{focus1}) is at most $M_4 (1 +
(nt)^{{r/2}} )$.
\end{pf*}
\end{appendix}

\section*{Acknowledgments}

The authors would like to thank Ton Dieker, Johan van Leeuwaarden, Josh
Reed, Ward Whitt and Bert Zwart for their helpful discussions and
insights. The authors especially thank Ton Dieker for his insights
into the large deviations properties of suprema of Gaussian processes.

%suskaldyti doi

% imsref loaded by lrinkeviciute, 2013-02-28 08:37:34
% imsref loaded by lrinkeviciute, 2013-02-28 08:42:27
% imsref loaded by lrinkeviciute, 2013-02-28 08:49:13

\printaddresses


\begin{thebibliography}{50}
% BibTex style file: ims.bst, 2013-01-28
% Default style options (sort=0,type=number).
% Used options (sort=1,type=number).

%b1 ###
\bibitem{AS72}
\begin{bbook}[auto:STB|2013/02/26|09:05:15]
\bauthor{\bsnm{Abramowitz},~\bfnm{M.}\binits{M.}} \AND
  \bauthor{\bsnm{Stegun},~\bfnm{I.}\binits{I.}}
(\byear{1972}).
\btitle{Handbook of Mathematical Functions}.
\bpublisher{Dover}, \blocation{Mineola, NY}.
\bptok{imsref}%
\end{bbook}
\endbibitem

%b2 ###
\bibitem{A90}
\begin{bbook}[mr]
\bauthor{\bsnm{Adler},~\bfnm{Robert~J.}\binits{R.~J.}}
(\byear{1990}).
\btitle{An Introduction to Continuity, Extrema, and Related Topics for General
  {G}aussian Processes}.
\bseries{Institute of Mathematical Statistics Lecture Notes---Monograph Series}
\bvolume{12}.
\bpublisher{IMS}, \blocation{Hayward, CA}.
\bid{mr={1088478}}
\bptok{imsref}%
\end{bbook}
\endbibitem

%b3 ###
\bibitem{AAM07}
\begin{barticle}[auto:STB|2013/02/26|09:05:15]
\bauthor{\bsnm{Aksin},~\bfnm{Z.}\binits{Z.}},
  \bauthor{\bsnm{Armory},~\bfnm{M.}\binits{M.}} \AND
  \bauthor{\bsnm{Mehrotra},~\bfnm{V.}\binits{V.}}
(\byear{2007}).
\btitle{The modern call center: A~multi-disciplinary perspective on operations
  management research}.
\bjournal{Production and Operations Management}
\bvolume{16}
\bpages{665--688}.
\bptok{imsref}%
\end{barticle}
\endbibitem

%b4 ###
\bibitem{A03}
\begin{bbook}[mr]
\bauthor{\bsnm{Asmussen},~\bfnm{S{\o}ren}\binits{S.}}
(\byear{2003}).
\btitle{Applied Probability and Queues:
Stochastic Modelling and Applied Probability},
\bedition{2nd} ed.
\bseries{Applications of Mathematics (New York)}
\bvolume{51}.
\bpublisher{Springer}, \blocation{New York}.
\bid{mr={1978607}}
\bptok{imsref}%
\end{bbook}
\endbibitem

%b5 ###
\bibitem{B99}
\begin{bbook}[mr]
\bauthor{\bsnm{Billingsley},~\bfnm{Patrick}\binits{P.}}
(\byear{1999}).
\btitle{Convergence of Probability Measures},
\bedition{2nd} ed.
\bpublisher{Wiley}, \blocation{New York}.
\bid{doi={10.1002/9780470316962}, mr={1700749}}
\bptok{imsref}%
\end{bbook}
\endbibitem

%b6 ###
\bibitem{B65}
\begin{barticle}[mr]
\bauthor{\bsnm{Borovkov},~\bfnm{A.~A.}\binits{A.~A.}}
(\byear{1965}).
\btitle{Some limit theorems in the theory of mass service. II}.
\bjournal{Theory Probab. Appl.}
\bvolume{10}
\bpages{375--400}.
\bptok{imsref}%
\end{barticle}
\endbibitem

%b7 ###
\bibitem{CTK94}
\begin{barticle}[mr]
\bauthor{\bsnm{Chang},~\bfnm{Cheng-Shang}\binits{C.-S.}},
  \bauthor{\bsnm{Thomas},~\bfnm{Joy~A.}\binits{J.~A.}} \AND
  \bauthor{\bsnm{Kiang},~\bfnm{Shaw-Hwa}\binits{S.-H.}}
(\byear{1994}).
\btitle{On the stability of open networks: A unified approach by stochastic
  dominance}.
\bjournal{Queueing Systems Theory Appl.}
\bvolume{15}
\bpages{239--260}.
\bid{doi={10.1007/BF01189239}, issn={0257-0130}, mr={1266795}}
\bptok{imsref}%
\end{barticle}
\endbibitem

%b8 ###
\bibitem{CHL79}
\begin{barticle}[mr]
\bauthor{\bsnm{Chow},~\bfnm{Y.~S.}\binits{Y.~S.}},
  \bauthor{\bsnm{Hsiung},~\bfnm{Chao~A.}\binits{C.~A.}} \AND
  \bauthor{\bsnm{Lai},~\bfnm{T.~L.}\binits{T.~L.}}
(\byear{1979}).
\btitle{Extended renewal theory and moment convergence in {A}nscombe's
  theorem}.
\bjournal{Ann. Probab.}
\bvolume{7}
\bpages{304--318}.
\bid{issn={0091-1798}, mr={0525056}}
\bptok{imsref}%
\end{barticle}
\endbibitem

%b9 ###
\bibitem{Cox70}
\begin{bbook}[mr]
\bauthor{\bsnm{Cox},~\bfnm{D.~R.}\binits{D.~R.}}
(\byear{1962}).
\btitle{Renewal Theory}.
\bpublisher{Methuen}, \blocation{London}.
\bid{mr={0153061}}
\bptnote{check year}%
\bptok{imsref}%
\end{bbook}
\endbibitem

%b10 ###
\bibitem{DH10}
\begin{barticle}[mr]
\bauthor{\bsnm{Dai},~\bfnm{J.~G.}\binits{J.~G.}} \AND
  \bauthor{\bsnm{He},~\bfnm{Shuangchi}\binits{S.}}
(\byear{2010}).
\btitle{Customer abandonment in many-server queues}.
\bjournal{Math. Oper. Res.}
\bvolume{35}
\bpages{347--362}.
\bid{doi={10.1287/moor.1100.0443}, issn={0364-765X}, mr={2674724}}
\bptok{imsref}%
\end{barticle}
\endbibitem

%b11 ###
\bibitem{DHT10}
\begin{barticle}[mr]
\bauthor{\bsnm{Dai},~\bfnm{J.~G.}\binits{J.~G.}},
  \bauthor{\bsnm{He},~\bfnm{Shuangchi}\binits{S.}} \AND
  \bauthor{\bsnm{Tezcan},~\bfnm{Tolga}\binits{T.}}
(\byear{2010}).
\btitle{Many-server diffusion limits for {$G/Ph/n+GI$} queues}.
\bjournal{Ann. Appl. Probab.}
\bvolume{20}
\bpages{1854--1890}.
\bid{doi={10.1214/09-AAP674}, issn={1050-5164}, mr={2724423}}
\bptok{imsref}%
\end{barticle}
\endbibitem

%b12 ###
\bibitem{Deb99}
\begin{barticle}[mr]
\bauthor{\bsnm{Debicki},~\bfnm{Krzysztof}\binits{K.}}
(\byear{1999}).
\btitle{A note on {LDP} for supremum of {G}aussian processes over infinite
  horizon}.
\bjournal{Statist. Probab. Lett.}
\bvolume{44}
\bpages{211--219}.
\bid{doi={10.1016/S0167-7152(99)00011-5}, issn={0167-7152}, mr={1711641}}
\bptok{imsref}%
\end{barticle}
\endbibitem

%b13 ###
\bibitem{DMR98}
\begin{barticle}[mr]
\bauthor{\bsnm{Debicki},~\bfnm{Krzysztof}\binits{K.}},
  \bauthor{\bsnm{Michna},~\bfnm{Zbigniew}\binits{Z.}} \AND
  \bauthor{\bsnm{Rolski},~\bfnm{Tomasz}\binits{T.}}
(\byear{1998}).
\btitle{On the supremum from {G}aussian processes over infinite horizon}.
\bjournal{Probab. Math. Statist.}
\bvolume{18}
\bpages{83--100}.
\bid{issn={0208-4147}, mr={1644033}}
\bptok{imsref}%
\end{barticle}
\endbibitem

%b14 ###
\bibitem{D05}
\begin{barticle}[mr]
\bauthor{\bsnm{Dieker},~\bfnm{A.~B.}\binits{A.~B.}}
(\byear{2005}).
\btitle{Conditional limit theorem for queues with {G}aussian input, a weak
  convergence approach}.
\bjournal{Stochastic Process. Appl.}
\bvolume{115}
\bpages{849--873}.
\bid{doi={10.1016/j.spa.2004.11.008}, issn={0304-4149}, mr={2132601}}
\bptok{imsref}%
\end{barticle}
\endbibitem

%b15 ###
\bibitem{D44}
\begin{barticle}[mr]
\bauthor{\bsnm{Doob},~\bfnm{J.~L.}\binits{J.~L.}}
(\byear{1944}).
\btitle{The elementary {G}aussian processes}.
\bjournal{Ann. Math. Statist.}
\bvolume{15}
\bpages{229--282}.
\bid{issn={0003-4851}, mr={0010931}}
\bptok{imsref}%
\end{barticle}
\endbibitem

%b16 ###
\bibitem{DO95}
\begin{barticle}[mr]
\bauthor{\bsnm{Duffield},~\bfnm{N.~G.}\binits{N.~G.}} \AND
  \bauthor{\bsnm{O'Connell},~\bfnm{Neil}\binits{N.}}
(\byear{1995}).
\btitle{Large deviations and overflow probabilities for the general
  single-server queue, with applications}.
\bjournal{Math. Proc. Cambridge Philos. Soc.}
\bvolume{118}
\bpages{363--374}.
\bid{doi={10.1017/S0305004100073709}, issn={0305-0041}, mr={1341797}}
\bptok{imsref}%
\end{barticle}
\endbibitem

%b17 ###
\bibitem{DLS03}
\begin{barticle}[mr]
\bauthor{\bsnm{Duffy},~\bfnm{Ken}\binits{K.}},
  \bauthor{\bsnm{Lewis},~\bfnm{John~T.}\binits{J.~T.}} \AND
  \bauthor{\bsnm{Sullivan},~\bfnm{Wayne~G.}\binits{W.~G.}}
(\byear{2003}).
\btitle{Logarithmic asymptotics for the supremum of a stochastic process}.
\bjournal{Ann. Appl. Probab.}
\bvolume{13}
\bpages{430--445}.
\bid{doi={10.1214/aoap/1050689587}, issn={1050-5164}, mr={1970270}}
\bptok{imsref}%
\end{barticle}
\endbibitem

%b18 ###
\bibitem{E48}
\begin{bbook}[auto:STB|2013/02/26|09:05:15]
\bauthor{\bsnm{Erlang},~\bfnm{A.~K.}\binits{A.~K.}}
(\byear{1948}).
\btitle{On the Rational determination of the Number of Circuits}.
\bpublisher{The Copenhagen Telephone Company}, \blocation{Copenhagen}.
\bptok{imsref}%
\end{bbook}
\endbibitem

%b19 ###
\bibitem{GM08}
\begin{barticle}[mr]
\bauthor{\bsnm{Gamarnik},~\bfnm{David}\binits{D.}} \AND
  \bauthor{\bsnm{Mom{\v{c}}ilovi{\'c}},~\bfnm{Petar}\binits{P.}}
(\byear{2008}).
\btitle{Steady-state analysis of a multiserver queue in the {H}alfin--{W}hitt
  regime}.
\bjournal{Adv. in Appl. Probab.}
\bvolume{40}
\bpages{548--577}.
\bid{issn={0001-8678}, mr={2433709}}
\bptok{imsref}%
\end{barticle}
\endbibitem

%b20 ###
\bibitem{GS11a}
\begin{barticle}[mr]
\bauthor{\bsnm{Gamarnik},~\bfnm{David}\binits{D.}} \AND
  \bauthor{\bsnm{Stolyar},~\bfnm{Alexander~L.}\binits{A.~L.}}
(\byear{2012}).
\btitle{Multiclass multiserver queueing system in the {H}alfin--{W}hitt heavy
  traffic regime: Asymptotics of the stationary distribution}.
\bjournal{Queueing Syst.}
\bvolume{71}
\bpages{25--51}.
\bid{doi={10.1007/s11134-012-9294-x}, issn={0257-0130}, mr={2925789}}
\bptok{imsref}%
\end{barticle}
\endbibitem

%b21 ###
\bibitem{HW81}
\begin{barticle}[mr]
\bauthor{\bsnm{Halfin},~\bfnm{Shlomo}\binits{S.}} \AND
  \bauthor{\bsnm{Whitt},~\bfnm{Ward}\binits{W.}}
(\byear{1981}).
\btitle{Heavy-traffic limits for queues with many exponential servers}.
\bjournal{Oper. Res.}
\bvolume{29}
\bpages{567--588}.
\bid{doi={10.1287/opre.29.3.567}, issn={0030-364X}, mr={0629195}}
\bptok{imsref}%
\end{barticle}
\endbibitem

%b22 ###
\bibitem{IR78}
\begin{bbook}[mr]
\bauthor{\bsnm{Ibragimov},~\bfnm{Il'dar~Abdulovich}\binits{I.~A.}} \AND
  \bauthor{\bsnm{Rozanov},~\bfnm{Y.~A.}\binits{Y.~A.}}
(\byear{1978}).
\btitle{Gaussian Random Processes}.
\bseries{Applications of Mathematics}
\bvolume{9}.
\bpublisher{Springer}, \blocation{New York}.
\bid{mr={0543837}}
\bptok{imsref}%
\end{bbook}
\endbibitem

%b23 ###
\bibitem{W70}
\begin{barticle}[mr]
\bauthor{\bsnm{Iglehart},~\bfnm{Donald~L.}\binits{D.~L.}} \AND
  \bauthor{\bsnm{Whitt},~\bfnm{Ward}\binits{W.}}
(\byear{1970}).
\btitle{Multiple channel queues in heavy traffic. {I}}.
\bjournal{Adv. in Appl. Probab.}
\bvolume{2}
\bpages{150--177}.
\bid{issn={0001-8678}, mr={0266331}}
\bptok{imsref}%
\end{barticle}
\endbibitem

%b24 ###
\bibitem{J74}
\begin{barticle}[mr]
\bauthor{\bsnm{Jagerman},~\bfnm{D.~L.}\binits{D.~L.}}
(\byear{1974}).
\btitle{Some properties of the {E}rlang loss function}.
\bjournal{Bell System Tech. J.}
\bvolume{53}
\bpages{525--551}.
\bid{issn={0005-8580}, mr={0394936}}
\bptok{imsref}%
\end{barticle}
\endbibitem

%b25 ###
\bibitem{JMM04}
\begin{barticle}[mr]
\bauthor{\bsnm{Jelenkovi{\'c}},~\bfnm{Predrag}\binits{P.}},
  \bauthor{\bsnm{Mandelbaum},~\bfnm{Avishai}\binits{A.}} \AND
  \bauthor{\bsnm{Mom{\v{c}}ilovi{\'c}},~\bfnm{Petar}\binits{P.}}
(\byear{2004}).
\btitle{Heavy traffic limits for queues with many deterministic servers}.
\bjournal{Queueing Syst.}
\bvolume{47}
\bpages{53--69}.
\bid{doi={10.1023/B:QUES.0000032800.52494.51}, issn={0257-0130}, mr={2074672}}
\bptok{imsref}%
\end{barticle}
\endbibitem

%b26 ###
\bibitem{kang2012asymptotic}
\begin{barticle}[mr]
\bauthor{\bsnm{Kang},~\bfnm{Weining}\binits{W.}} \AND
  \bauthor{\bsnm{Ramanan},~\bfnm{Kavita}\binits{K.}}
(\byear{2012}).
\btitle{Asymptotic approximations for stationary distributions of many-server
  queues with abandonment}.
\bjournal{Ann. Appl. Probab.}
\bvolume{22}
\bpages{477--521}.
\bid{doi={10.1214/10-AAP738}, issn={1050-5164}, mr={2953561}}
\bptok{imsref}%
\end{barticle}
\endbibitem

%b27 ###
\bibitem{RK10b}
\begin{barticle}[auto:STB|2013/02/26|09:05:15]
\bauthor{\bsnm{Kaspi},~\bfnm{H.}\binits{H.}} \AND
\bauthor{\bsnm{Ramanan},~\bfnm{K.}\binits{K.}}
(\byear{2013}).
\btitle{SPDE limits of many-server queues}.
\bjournal{Ann. Appl. Probab.}
\bvolume{23}
\bpages{145--229}.
\bptok{imsref}%
\end{barticle}
\endbibitem

%b28 ###
\bibitem{KS06}
\begin{barticle}[mr]
\bauthor{\bsnm{Kella},~\bfnm{Offer}\binits{O.}} \AND
  \bauthor{\bsnm{Stadje},~\bfnm{Wolfgang}\binits{W.}}
(\byear{2006}).
\btitle{Superposition of renewal processes and an application to multi-server
  queues}.
\bjournal{Statist. Probab. Lett.}
\bvolume{76}
\bpages{1914--1924}.
\bid{doi={10.1016/j.spl.2006.04.041}, issn={0167-7152}, mr={2271187}}
\bptok{imsref}%
\end{barticle}
\endbibitem

%b29 ###
\bibitem{MM08}
\begin{barticle}[mr]
\bauthor{\bsnm{Mandelbaum},~\bfnm{Avishai}\binits{A.}} \AND
  \bauthor{\bsnm{Mom{\v{c}}ilovi{\'c}},~\bfnm{Petar}\binits{P.}}
(\byear{2008}).
\btitle{Queues with many servers: The virtual waiting-time process in the {QED}
  regime}.
\bjournal{Math. Oper. Res.}
\bvolume{33}
\bpages{561--586}.
\bid{doi={10.1287/moor.1070.0310}, issn={0364-765X}, mr={2442642}}
\bptok{imsref}%
\end{barticle}
\endbibitem

%b30 ###
\bibitem{marcus2006markov}
\begin{bbook}[mr]
\bauthor{\bsnm{Marcus},~\bfnm{Michael~B.}\binits{M.~B.}} \AND
  \bauthor{\bsnm{Rosen},~\bfnm{Jay}\binits{J.}}
(\byear{2006}).
\btitle{Markov Processes, {G}aussian Processes, and Local Times}.
\bseries{Cambridge Studies in Advanced Mathematics}
\bvolume{100}.
\bpublisher{Cambridge Univ. Press}, \blocation{Cambridge}.
\bid{doi={10.1017/CBO9780511617997}, mr={2250510}}
\bptok{imsref}%
\end{bbook}
\endbibitem

%b31 ###
\bibitem{P65}
\begin{bbook}[auto:STB|2013/02/26|09:05:15]
\bauthor{\bsnm{Prabhu},~\bfnm{N.}\binits{N.}}
(\byear{1965}).
\btitle{Stochastic Processes}.
\bpublisher{World Scientific}, \blocation{Singapore}.
\bptok{imsref}%
\end{bbook}
\endbibitem

%b32 ###
\bibitem{Prot80}
\begin{barticle}[mr]
\bauthor{\bsnm{Protter},~\bfnm{Philip}\binits{P.}}
(\byear{1980}).
\btitle{Stochastic differential equations with jump reflection at the
  boundary}.
\bjournal{Stochastics}
\bvolume{3}
\bpages{193--201}.
\bid{doi={10.1080/17442508008833144}, issn={0090-9491}, mr={0573203}}
\bptok{imsref}%
\end{barticle}
\endbibitem

%b33 ###
\bibitem{PR10}
\begin{barticle}[mr]
\bauthor{\bsnm{Puhalskii},~\bfnm{Anatolii~A.}\binits{A.~A.}} \AND
  \bauthor{\bsnm{Reed},~\bfnm{Josh~E.}\binits{J.~E.}}
(\byear{2010}).
\btitle{On many-server queues in heavy traffic}.
\bjournal{Ann. Appl. Probab.}
\bvolume{20}
\bpages{129--195}.
\bid{doi={10.1214/09-AAP604}, issn={1050-5164}, mr={2582645}}
\bptok{imsref}%
\end{barticle}
\endbibitem

%b34 ###
\bibitem{PR00b}
\begin{barticle}[mr]
\bauthor{\bsnm{Puhalskii},~\bfnm{A.~A.}\binits{A.~A.}} \AND
  \bauthor{\bsnm{Reiman},~\bfnm{M.~I.}\binits{M.~I.}}
(\byear{2000}).
\btitle{The multiclass {$GI/PH/N$} queue in the {H}alfin--{W}hitt regime}.
\bjournal{Adv. in Appl. Probab.}
\bvolume{32}
\bpages{564--595}.
\bid{doi={10.1239/aap/1013540179}, issn={0001-8678}, mr={1778580}}
\bptok{imsref}%
\end{barticle}
\endbibitem

%b35 ###
\bibitem{R07b}
\begin{bmisc}[mr]
\bauthor{\bsnm{Reed},~\bfnm{Josh}\binits{J.}}
(\byear{2007}).
\bhowpublished{The $G/GI/N$ queue in the Halfin--Whitt regime
II: Idle time system equations.  Working paper}.
\bptok{imsref}%
\end{bmisc}
\endbibitem

%b36 ###
\bibitem{R09}
\begin{barticle}[mr]
\bauthor{\bsnm{Reed},~\bfnm{Josh}\binits{J.}}
(\byear{2009}).
\btitle{The {$G/GI/N$} queue in the {H}alfin--{W}hitt regime}.
\bjournal{Ann. Appl. Probab.}
\bvolume{19}
\bpages{2211--2269}.
\bid{doi={10.1214/09-AAP609}, issn={1050-5164}, mr={2588244}}
\bptok{imsref}%
\end{barticle}
\endbibitem

%b37 ###
\bibitem{R37}
\begin{barticle}[auto:STB|2013/02/26|09:05:15]
\bauthor{\bsnm{Riordan},~\bfnm{J.}\binits{J.}}
(\byear{1937}).
\btitle{Moment recurrence relations for binomial, Poisson and hypergeometric
  frequency distributions}.
\bjournal{Ann. Math. Statist.}
\bvolume{8}
\bpages{103--111}.
\bptok{imsref}%
\end{barticle}
\endbibitem

%b38 ###
\bibitem{Ross96}
\begin{bbook}[mr]
\bauthor{\bsnm{Ross},~\bfnm{Sheldon~M.}\binits{S.~M.}}
(\byear{1996}).
\btitle{Stochastic Processes},
\bedition{2nd} ed.
\bpublisher{Wiley}, \blocation{New York}.
\bid{mr={1373653}}
\bptok{imsref}%
\end{bbook}
\endbibitem

%b39 ###
\bibitem{SY89}
\begin{barticle}[mr]
\bauthor{\bsnm{Shanthikumar},~\bfnm{J.~George}\binits{J.~G.}} \AND
  \bauthor{\bsnm{Yao},~\bfnm{David~D.}\binits{D.~D.}}
(\byear{1989}).
\btitle{Stochastic monotonicity in general queueing networks}.
\bjournal{J. Appl. Probab.}
\bvolume{26}
\bpages{413--417}.
\bid{issn={0021-9002}, mr={1000301}}
\bptok{imsref}%
\end{barticle}
\endbibitem

%b40 ###
\bibitem{Sk61}
\begin{barticle}[auto:STB|2013/02/26|09:05:15]
\bauthor{\bsnm{Skorokhod},~\bfnm{A.~V.}\binits{A.~V.}}
(\byear{1961}).
\btitle{Stochastic equations for diffusions in a bounded region}.
\bjournal{Theory Probab. Appl.}
\bvolume{6}
\bpages{264--274}.
\bptok{imsref}%
\end{barticle}
\endbibitem

%b41 ###
\bibitem{S63}
\begin{barticle}[mr]
\bauthor{\bsnm{Stone},~\bfnm{Charles}\binits{C.}}
(\byear{1963}).
\btitle{Limit theorems for random walks, birth and death processes, and
  diffusion processes}.
\bjournal{Illinois J. Math.}
\bvolume{7}
\bpages{638--660}.
\bid{issn={0019-2082}, mr={0158440}}
\bptok{imsref}%
\end{barticle}
\endbibitem

%b42 ###
\bibitem{Stoyan83}
\begin{bbook}[mr]
\bauthor{\bsnm{Stoyan},~\bfnm{Dietrich}\binits{D.}}
(\byear{1983}).
\btitle{Comparison Methods for Queues and Other Stochastic Models}.
\bpublisher{Wiley}, \blocation{Chichester}.
\bid{mr={0754339}}
\bptok{imsref}%
\end{bbook}
\endbibitem

%b43 ###
\bibitem{Stight99}
\begin{barticle}[mr]
\bauthor{\bsnm{Szczotka},~\bfnm{W{\l}adys{\l}aw}\binits{W.}}
(\byear{1999}).
\btitle{Tightness of the stationary waiting time in heavy traffic}.
\bjournal{Adv. in Appl. Probab.}
\bvolume{31}
\bpages{788--794}.
\bid{doi={10.1239/aap/1029955204}, issn={0001-8678}, mr={1742694}}
\bptok{imsref}%
\end{barticle}
\endbibitem

%b44 ###
\bibitem{Tak62}
\begin{bbook}[mr]
\bauthor{\bsnm{Tak{\'a}cs},~\bfnm{Lajos}\binits{L.}}
(\byear{1962}).
\btitle{Introduction to the Theory of Queues}.
\bpublisher{Oxford Univ. Press}, \blocation{New York}.
\bid{mr={0133880}}
\bptok{imsref}%
\end{bbook}
\endbibitem

%b45 ###
\bibitem{W81c}
\begin{barticle}[mr]
\bauthor{\bsnm{Whitt},~\bfnm{Ward}\binits{W.}}
(\byear{1981}).
\btitle{Comparing counting processes and queues}.
\bjournal{Adv. in Appl. Probab.}
\bvolume{13}
\bpages{207--220}.
\bid{doi={10.2307/1426475}, issn={0001-8678}, mr={0595895}}
\bptok{imsref}%
\end{barticle}
\endbibitem

%b46 ###
\bibitem{Whitt85}
\begin{barticle}[mr]
\bauthor{\bsnm{Whitt},~\bfnm{Ward}\binits{W.}}
(\byear{1985}).
\btitle{Queues with superposition arrival processes in heavy traffic}.
\bjournal{Stochastic Process. Appl.}
\bvolume{21}
\bpages{81--91}.
\bid{doi={10.1016/0304-4149(85)90378-3}, issn={0304-4149}, mr={0834989}}
\bptok{imsref}%
\end{barticle}
\endbibitem

%b47 ###
\bibitem{W00}
\begin{barticle}[mr]
\bauthor{\bsnm{Whitt},~\bfnm{Ward}\binits{W.}}
(\byear{2000}).
\btitle{The impact of a heavy-tailed service-time distribution upon the
  {$M/GI/s$} waiting-time distribution}.
\bjournal{Queueing Systems Theory Appl.}
\bvolume{36}
\bpages{71--87}.
\bid{doi={10.1023/A:1019143505968}, issn={0257-0130}, mr={1806961}}
\bptok{imsref}%
\end{barticle}
\endbibitem

%b48 ###
\bibitem{W02}
\begin{bbook}[mr]
\bauthor{\bsnm{Whitt},~\bfnm{Ward}\binits{W.}}
(\byear{2002}).
\btitle{Stochastic-Process Limits:
An Introduction to Stochastic-Process Limits and their Application to
  Queues}.
\bpublisher{Springer}, \blocation{New York}.
\bid{mr={1876437}}
\bptok{imsref}%
\end{bbook}
\endbibitem

%b49 ###
\bibitem{W60}
\begin{barticle}[mr]
\bauthor{\bsnm{Whittle},~\bfnm{P.}\binits{P.}}
(\byear{1960}).
\btitle{Bounds for the moments of linear and quadratic forms in independent
  variables}.
\bjournal{Theory Probab. Appl.}
\bvolume{5}
\bpages{302--305}.
\bptok{imsref}%
\end{barticle}
\endbibitem

\end{thebibliography}
\end{document}